\documentclass[a4paper]{article}
\usepackage{amsmath,amsthm,amsfonts,thmtools}
\usepackage{amssymb}
\usepackage[numbers, compress]{natbib}
\usepackage[cmintegrals]{newtxmath}
\usepackage{graphicx}
\usepackage{listings}
\usepackage[font=footnotesize]{caption}
\usepackage[font=scriptsize]{subcaption}
\usepackage{euscript}
\usepackage{units}
\usepackage{enumerate}
\usepackage{xcolor}
\usepackage{todonotes}
\usepackage{booktabs}
\usepackage{tikz}
\usepackage{comment}
 
\DeclareMathAlphabet{\pazocal}{OMS}{zplm}{m}{n}

\usepackage[bookmarks=true,hidelinks]{hyperref}
\usepackage{bbm}
\usepackage{multirow}
\usepackage[normalem]{ulem} % provides sout

\usepackage{mathtools}
\usepackage{subcaption} 
\usepackage{graphicx}
\usepackage{pgfplots}
\usepackage{float}
\usepackage{algpseudocode}
\usepackage{bbold}
\usepackage{microtype}
\usepackage{rotating}
\usepackage{enumitem}
\usepackage{bm}
\usepackage{hyperref}

\numberwithin{equation}{section}
\newtheorem{theorem}{Theorem}[section]

\newtheorem{lemma}[theorem]{Lemma}

\newtheorem{algorithm}[theorem]{Algorithm}

\numberwithin{equation}{section}

\theoremstyle{definition}

\newtheoremstyle{myremarkstyle}{}{}{}{}{\bfseries}{.}{ }{}
\theoremstyle{myremarkstyle}
\declaretheorem[name=Remark,qed=$\blacksquare$,numberlike=theorem]{remark}

\makeatletter
\newcommand*{\intavg}{%
  \mint@l{-}{}%
}
\newcommand*{\mint@l}[2]{%
  \@ifnextchar\limits{%
    \mint@l{#1}%
  }{%
    \@ifnextchar\nolimits{%
      \mint@l{#1}%
    }{%
      \@ifnextchar\displaylimits{%
        \mint@l{#1}%
      }{%
        \mint@s{#2}{#1}%
      }%
    }%
  }%
}
\newcommand*{\mint@s}[2]{%
  \@ifnextchar_{%
    \mint@sub{#1}{#2}%
  }{%
    \@ifnextchar^{%
      \mint@sup{#1}{#2}%
    }{%
      \mint@{#1}{#2}{}{}%
    }%
  }%
}
\def\mint@sub#1#2_#3{%
  \@ifnextchar^{%
    \mint@sub@sup{#1}{#2}{#3}%
  }{%
    \mint@{#1}{#2}{#3}{}%
  }%
}
\def\mint@sup#1#2^#3{%
  \@ifnextchar_{%
    \mint@sub@sup{#1}{#2}{#3}%
  }{%
    \mint@{#1}{#2}{}{#3}%
  }%
}
\def\mint@sub@sup#1#2#3^#4{%
  \mint@{#1}{#2}{#3}{#4}%
}
\def\mint@sup@sub#1#2#3_#4{%
  \mint@{#1}{#2}{#4}{#3}%
}
\newcommand*{\mint@}[4]{%
  \mathop{}%
  \mkern-\thinmuskip
  \mathchoice{%
    \mint@@{#1}{#2}{#3}{#4}%
        \displaystyle\textstyle\scriptstyle
  }{%
    \mint@@{#1}{#2}{#3}{#4}%
        \textstyle\scriptstyle\scriptstyle
  }{%
    \mint@@{#1}{#2}{#3}{#4}%
        \scriptstyle\scriptscriptstyle\scriptscriptstyle
  }{%
    \mint@@{#1}{#2}{#3}{#4}%
        \scriptscriptstyle\scriptscriptstyle\scriptscriptstyle
  }%
  \mkern-\thinmuskip
  \int#1%
  \ifx\\#3\\\else_{#3}\fi
  \ifx\\#4\\\else^{#4}\fi  
}
\newcommand*{\mint@@}[7]{%
  \begingroup
    \sbox0{$#5\int\m@th$}%
    \sbox2{$#5\int_{}\m@th$}%
    \dimen2=\wd0 %
    \let\mint@limits=#1\relax
    \ifx\mint@limits\relax
      \sbox4{$#5\int_{\kern1sp}^{\kern1sp}\m@th$}%
      \ifdim\wd4>\wd2 %
        \let\mint@limits=\nolimits
      \else
        \let\mint@limits=\limits
      \fi
    \fi
    \ifx\mint@limits\displaylimits
      \ifx#5\displaystyle
        \let\mint@limits=\limits
      \fi
    \fi
    \ifx\mint@limits\limits
      \sbox0{$#7#3\m@th$}%
      \sbox2{$#7#4\m@th$}%
      \ifdim\wd0>\dimen2 %
        \dimen2=\wd0 %
      \fi
      \ifdim\wd2>\dimen2 %
        \dimen2=\wd2 %
      \fi
    \fi
    \rlap{%
      $#5%
        \vcenter{%
          \hbox to\dimen2{%
            \hss
            $#6{#2}\m@th$%
            \hss
          }%
        }%
      $%
    }%
  \endgroup
}

\def\XXint#1#2#3{{\setbox0=\hbox{$#1{#2#3}{\int}$ }
		\vcenter{\hbox{$#2#3$ }}\kern-.6\wd0}}

\DeclareMathOperator*{\argmin}{argmin}

\renewcommand{\geq}{\geqslant}
\renewcommand{\leq}{\leqslant}

\renewcommand{\epsilon}{\varepsilon}
\renewcommand{\phi}{\varphi}

\newcommand{\R}{\mathbb{R}}
\newcommand{\N}{\mathbb{N}}

\newcommand{\U}{{\bf U}}		% Phase space

\newcommand{\vel}{\mathbf{u}}

\newcommand{\F}{{\bf F}}

\newcommand{\map}{\EuScript{L}}
\newcommand{\train}{\EuScript{S}}

\newcommand{\reg}{\EuScript{R}}
\newcommand{\er}{\EuScript{E}}

\newcommand{\cost}{\EuScript{C}}

\newcommand{\gl}{\EuScript{G}}
\newcommand{\rdl}{\EuScript{Y}}
\newcommand{\Ord}{{\mathcal O}}

\begin{document}

\date{\today}

\title{Iterative Surrogate Model Optimization (ISMO): \\
An active learning algorithm for PDE constrained optimization with deep neural networks.}

\author{Kjetil O. Lye \thanks{Mathematics and Cybernetics, SINTEF\newline Forskningsveien 1, 0373 Oslo, Norway}, Siddhartha Mishra \thanks{Seminar for Applied Mathematics (SAM), D-Math \newline
  ETH Z\"urich, R\"amistrasse 101, 
  Z\"urich-8092, Switzerland}, Deep Ray \thanks{Department of Aerospace and Mechanical Engineering, University of Southern California, Los Angeles, U.S.A} and Praveen Chandrasekhar \thanks{TIFR Centre for Applicable Mathematics, Bangalore-560065, India.}}

\date{\today}

\maketitle
\begin{abstract}
We present a novel \emph{active learning} algorithm, termed as \emph{iterative surrogate model optimization} (ISMO), for robust and efficient numerical approximation of PDE constrained optimization problems. This algorithm is based on deep neural networks and its key feature is the iterative selection of training data through a feedback loop between deep neural networks and any underlying standard optimization algorithm. Under suitable hypotheses, we show that the resulting optimizers converge exponentially fast (and with exponentially decaying variance), with respect to increasing number of training samples. Numerical examples for optimal control, parameter identification and shape optimization problems for PDEs are provided to validate the proposed theory and to illustrate that ISMO significantly outperforms a standard deep neural network based surrogate optimization algorithm.
\end{abstract}

\section{Introduction}
A large number of problems of interest in engineering reduce to the following \emph{optimization problem}: 
\begin{equation}
    \label{eq:opt1}
    {\rm Find}\quad \bar{y} = {\rm arg}\min\limits_{y \in Y} G\left(\map(y)\right),
\end{equation}
Here, $Y \subset \R^d$, the underlying map (\emph{observable}) $\map: Y \to Z$, with either $Z \subset \R^m$ or is a subset of an infinite-dimensional Banach space and $G: Z \to \R$ is the \emph{cost} or \emph{objective function}. We can think of $Y$
 as a design or control space, with any $y \in Y$ being a vector of design or control parameters, that steers the system to minimize the objective function. 
 
Many engineering systems of interest are modeled by partial differential equations (PDEs). Thus, evaluating the desired observable $\map$ and consequently, the objective function in \eqref{eq:opt1}, requires (approximate) solutions of the underlying PDE. The resulting problem is termed as a \emph{PDE constrained optimization} problem, \cite{BS1,TRL1} and references therein. 

A representative example for PDE constrained optimization is provided by the problem of designing the wing (or airfoils) of an aircraft in order to minimize drag~\cite{Reuther1995},~\cite{Mohammadi2009}. Here, the design variables are the parameters that describe the wing shape. The underlying PDEs are the compressible Euler or Navier-Stokes equations of gas dynamics and the objective function is the drag coefficient (while keeping the lift coefficient within a range) or the lift to drag ratio. Several other examples of PDE constrained optimization problems arise in fluid dynamics, solid mechanics and the geosciences \cite{BS1}. 

A wide variety of numerical methods have been devised to solve optimization problems in general, and PDE constrained optimization problems in particular \cite{BS1,TRL1}. A large class of methods are iterative and require evaluation of the gradients of the objective function in \eqref{eq:opt1}, with respect to the design (control) parameters $y$. These include first-order methods based on \emph{gradient descent} and second-order methods such as quasi-newton  and truncated Newton methods \cite{lbfgs}. Evaluating gradients in the context of PDE constrained optimization often involves the (numerical) solution of \emph{adjoints} (duals) of the underlying PDE~\cite{BS1}. Other types of optimization methods such as particle swarm optimization \cite{PSO} and genetic algorithms~\cite{GA} are \emph{gradient free} and are particularly suited for optimization, constrained by PDEs with rough solutions. 

Despite the well documented success of some of the aforementioned algorithms in the context of PDE constrained optimization, the solution of realistic optimization problems involving a large number of design parameters ($d \gg 1$ in \eqref{eq:opt1}) and with the underlying PDEs set in complex geometries in several space dimensions or describing multiscale, multiphysics phenomena, remains a considerable challenge. This is mostly on account of the fact that applying the above methods entails evaluating the objective function and hence the underlying map $\map$ and its gradients a large number (in the range of  $10^2-10^4$) of times. As a single call to the underlying PDE solver can be very expensive, computing the underlying PDEs (and their adjoints) multiple times could be prohibitively expensive.  

%%This computational cost is considerably accentuated when one is interested in \emph{robust optimization} \cite{SS1,SS2} and references therein. For these optimization problems, the underlying PDEs involve uncertain parameters and coefficients, which are modeled in a probabilistic manner. The resulting optimization problem requires optimatization of suitable statistical quantities of the underlying map $\map$ in \eqref{eq:opt1}. Applying the above mentioned algorithms in this context requires much larger number of calls to the underlying PDE solver (for the optimization as well as for evaluating statistical quantities of interest), making robust optimization, constrained by realistic PDEs, very challenging. 

One possible framework for reducing the computational cost of PDE constrained optimization problems is to use \emph{surrogates}~\cite{Forrester2008}. This approach consists of generating \emph{training data}, i.e. computing $\map(y)$, $\forall y \in \train$, with $\train \subset Y$ denoting a \emph{training set}. Then, a \emph{surrogate model} is constructed by designing a surrogate map, $\hat{\map}: Y \to Z$ such that $\hat{\map}(y) \approx \map(y)$, for all $y \in \train$. Finally, one runs a standard optimization algorithm such as gradient descent or its variants, while evaluating the functions and gradients in \eqref{eq:opt1}, with respect to the surrogate map $\hat{\map}$. This surrogate model will be effective as long as $\map \approx \hat{\map}$ in a suitable sense, for all $y \in Y$ and the cost of evaluating the surrogate map $\hat{\map}$ is significantly lower than the cost of evaluating the underlying map $\map$. Examples of such surrogate models include reduced order models \cite{ROMbook} and Gaussian process regression \cite{GPRbook}. 

A particularly attractive class of such surrogate models are \emph{deep neural networks} (DNNs) \cite{DLbook}, i.e. functions formed by multiple compositions of affine transformations and scalar non-linear activation functions. Deep neural networks have been extremely successful at diverse tasks in science and engineering \cite{DLnat} such as at image and text classification, computer vision, text and speech recognition, autonomous systems and robotics, game intelligence and even protein folding \cite{Dfold}. 

Deep neural networks are also being increasingly used in different contexts in scientific computing. A very incomplete list of the rapidly growing literature includes the use of deep neural networks to approximate solutions of PDEs by so-called physics informed neural networks (PINNs) \cite{Lag1,KAR1,KAR2,KAR4,MM1,MM2} and references therein, solutions of high-dimensional PDEs, particularly in finance \cite{E1,HEJ1,Jent1} and references therein, improving the efficiency of existing numerical methods for PDEs, for instance in \cite{INC,DR1,SM1} and references therein.

%%Many of the afore-mentioned papers in scientific computing use deep neural networks in the context of \emph{supervised learning} \cite{DLbook}, i.e. training the tuning parameters (weights and biases) of the neural network to minimize the so-called loss function (difference between the underlying map and the neural network surrogate on the training set in some suitable norm) with a stochastic gradient descent method.

A different approach is taken in recent papers \cite{LMR1,LMM1,MR1}, where the authors presented supervised deep learning algorithms to efficiently approximate \emph{observables} (quantities of interest) for solutions of PDEs and applied them to speedup existing (Quasi)-Monte Carlo algorithms for forward Uncertainty quantification (UQ). These papers demonstrated that deep neural networks can result in effective surrogates for observables in the context of PDEs. Given the preceding discussion, it is natural to take a step further and ask if the resulting surrogates can be used for PDE constrained optimization. This indeed constitutes the central premise of the current paper. 

Our first aim in this article is to present an algorithm for PDE constrained optimization that is based on combining standard (gradient based) optimization algorithms and deep neural network surrogates for the underlying maps in \eqref{eq:opt1}. The resulting algorithm, which we term \emph{DNNopt}, is described and examined, both theoretically (with suitable hypotheses on the underlying problem) and in numerical experiments. We find that although \emph{DNNopt} can be a viable algorithm for PDE constrained optimization and converges to a (local) minimum of the underlying optimization problem \eqref{eq:opt1}, it is not robust enough and can lead to a high \emph{variance} or sensitivity of the resulting (approximate) minima, with respect to starting values for the underlying optimization algorithm. A careful analysis reveals that a key cause for this high variance (large sensitivity) is the fact that the training set for the deep neural network surrogate is fixed \emph{a priori}, based on global approximation requirements. This training set may not necessarily represent the subset (in parameter space) of minimizers of the objective function in \eqref{eq:opt1} well, if at all. The resulting high variance can impede the efficiency of the algorithm and needs to be remedied.

To this end, the second and \emph{main aim} of this paper is to propose a novel \emph{active learning} procedure to \emph{iteratively} augment training sets for training a sequence of deep neural networks, each providing successively better approximation of the underlying minimizers of \eqref{eq:opt1}. The additions to the training sets in turn, are based on local minimizers of \eqref{eq:opt1}, identified by running standard optimization algorithms on the neural network surrogate at the previous step of iteration. This feedback between training neural networks for the observable $\map$ in \eqref{eq:opt1} and adding local optimizers as training points leads to the proposed algorithm, that we term as \emph{iterative surrogate model optimization (ISMO)}. Thus, ISMO can be thought of as an example of an \emph{active learning algorithm} \cite{AL}, where the learner (deep neural network) queries the teacher (standard optimization algorithm) to iteratively identify training data that provides a better approximation of local optima. We analyze ISMO theoretically  to find that the proposed algorithm is able to approximate minimizers of \eqref{eq:opt1}, with remarkably low variance or sensitivity with respect to starting values for the optimization algorithm. In particular, at least in restricted settings, ISMO converges exponentially fast to the underlying minima, with the variance also decaying exponentially fast. This behavior is observed in several numerical experiments. Thus, ISMO is shown to be a very effective framework for PDE constrained optimization. 

The rest of this paper is organized as follows: in section \ref{sec:2}, we put together preliminary material for this article. The DNNopt  and ISMO   algorithms for PDE constrained optimization are presented in sections \ref{sec:3} and \ref{sec:4}, respectively and numerical experiments are reported in section \ref{sec:5}. 
\section{Preliminaries}
\label{sec:2}
\subsection{The underlying constrained optimization problem}
The basis of our constrained optimization problem is the following time-dependent parametric PDE,
\begin{equation}
\label{eq:ppde}
\begin{aligned}
\partial_t \U(t,x,y) &= L\left(y, \U, \nabla_x \U, \nabla^2_x \U, \ldots \right), \quad \forall~(t,x,y) \in [0,T] \times D(y) \times Y, \\
\U(0,x,y) &= \overline{\U}(x,y), \quad \forall~ (x,y) \in D(y) \times Y, \\
L_{b} \U(t,x,y) &= \U_b (t,x,y), \quad \forall ~(t,x,y) \in [0,T] \times \partial D(y) \times Y.
\end{aligned}
\end{equation}
Here, $Y \subset \R^d$ is the underlying parameter space that describes the design (control) variables $y \in Y$. For simplicity of notation and exposition, we set $Y = [0,1]^d$ for the rest of this paper. The spatial domain is labeled as $y \rightarrow D(y) \subset \R^{d_s}$ and  $\U: [0,T] \times D \times Y  \rightarrow \R^{n}$ is the vector of unknowns. The differential operator $L$ is in a very generic form and can depend on the gradient and Hessian of $\U$, and possibly higher-order spatial derivatives. For instance, the heat equation as well as the Euler or Navier-Stokes equations of fluid dynamics are specific examples. Moreover, $L_b$ is a generic operator for imposing boundary conditions.

For the parametrized PDE \eqref{eq:ppde}, we define a generic form of the \emph{parameters to observable} map:
\begin{equation}
\label{eq:ptoob}
\map:Y\to \R^m, \; Y \ni y \mapsto \map(y) = \map(y,\U) \in \R^m.
\end{equation}

%\begin{equation}
%\label{eq:obsp}
%L_g(y,\U) := \int\limits_0^T\int\limits_{D(y)} \psi(x,t) g(\U(t,x,y)) dx dt, \quad {\rm for}\quad y \in Y.
%\end{equation} 
%Here, $\psi \in L^1_{{\rm loc}} (D(y) \times (0,T);\R^m)$ is a  \emph{test function} and $g \in C^s(\R^n;\R^m)$, for $s \geq 1$. Moreover, the product of two $m$-vectors in \eqref{eq:obsp} is taken componentwise. 

%For fixed functions $\psi,g$, we define the \emph{parameters to observable} map:
%\begin{equation}
%\label{eq:ptoob}
%\map:Y\to \R^m, \; Y \ni y \mapsto \map(y) = L_g(y,\U) \in \R^m,
%\end{equation}
%with $L_g$ being defined by \eqref{eq:obsp}. 

Finally, we define the cost (objective or goal) function $G$ as $G \in C^2(\R^m;\R)$. A very relevant example for such a cost function is given by,
\begin{equation}
    \label{eq:cost1}
    G(\map(y)):= |\map(y) - \bar{\map}|^p,
\end{equation}
for some target $\bar{\map} \in \R^m$, with $1 \leq p < \infty$ and $|.|$ denoting a vector norm. 

The resulting constrained optimization problem is given by
\begin{equation}
    \label{eq:opt}
    {\rm Find}\quad \bar{y} = {\rm arg}\min\limits_{y \in Y} G\left(\map(y)\right),
\end{equation}
with the observable defined by \eqref{eq:ptoob} and $C^2$-objective function, for instance the one defined in \eqref{eq:cost1}. 

We also denote, 
\begin{equation}
    \label{eq:cost2}
    \gl(y) = G(\map(y)), \quad \forall y \in Y,
\end{equation}
for notational convenience. 
\subsection{Standard optimization algorithms}
For definiteness, we restrict ourselves to the class of \emph{(quasi)-Newton} algorithms as our standard optimization algorithm in this paper. We assume that the cost function $\gl$ \eqref{eq:cost2} is such that $\gl \in C^2(Y)$ and we are interested in computing (approximate, local) minimizers of the optimization problem \eqref{eq:opt}. A generic quasi-Newton algorithm has the following form,
\begin{algorithm} 
\label{alg:qn} {\bf Quasi-Newton approximation of \eqref{eq:opt}} 
\begin{itemize}
\item [{\bf Inputs}:] Underlying cost function $\gl$ \eqref{eq:cost2}, starting value $y_0 \in Y$ and starting Hessian $B_0 \in \R^{d \times d}$, tolerance parameter $\epsilon$ 
\item [{\bf Goal}:] Find (local) minimizer for the optimization problem \eqref{eq:opt}. 
\item [{\bf At Step $k$}:] Given $y_k$, $B_k$
\begin{itemize}
    \item Find search direction $p_k \in \R^d$ by solving 
    \begin{equation}
        \label{eq:qn1}
        B_k p_k = -\nabla_y \gl(y_k)
    \end{equation}
    \item Perform a \emph{line search} (one-dimensional optimization):
    \begin{equation}
        \label{eq:qn2}
        \beta_k = {\rm arg}\min\limits_{\beta \in [0,1]} \gl(y_k + \beta p_k)
    \end{equation}
    \item Set $y_{k+1} = y_k + \beta_k p_k$
    \item Update the approximate Hessian,
    \begin{equation}
        \label{eq:qn3}
        B_{k+1} = {\mathcal F}\left(B_k, y_k, y_{k+1},\nabla_y \gl(y_k),\nabla_y \gl(y_{k+1}) \right)
    \end{equation}
\end{itemize}
\item If $|y_{k+1} - y_k| < \epsilon$, terminate the iteration, else continue.
\end{itemize}
\end{algorithm}
Different formulas for the update of the Hessian \eqref{eq:qn3} lead to different versions of the quasi-Newton algorithm \eqref{alg:qn}. For instance, the widely used BFGS algorithm \cite{lbfgs} uses an update formula \eqref{eq:qn3} such that the inversion step \eqref{eq:qn1} can be directly computed with the well-known Sherman-Morrison formula, from a recurrence relation. Other versions of quasi-Newton algorithms, in general truncated Newton algorithms, use iterative linear solvers for computing the solution to \eqref{eq:qn1}. Note that quasi-Newton methods do not explicitly require the Hessian of the objective function in \eqref{eq:opt}, but compute (estimate) it from the gradients in \eqref{eq:qn3}. Thus, only information on the gradient of the cost $\gl$ \eqref{eq:cost2} is required. 

We observe that implementing the algorithm \ref{alg:qn} in the specific context of the optimization problem \eqref{eq:opt}, constrained by the PDE \eqref{eq:ppde}, requires evaluating the map $\map$ and its gradients multiple times during each iteration. Thus, having a large number of iterations in algorithm \ref{alg:qn} entails a very high computational cost on account of the large number of calls to the underlying PDE solver. 
\subsection{Neural network surrogates}
As mentioned in the introduction, we will employ neural network based surrogates for the underlying observable $\map$ \eqref{eq:ptoob} in order to reduce the computational cost of employing standard optimization algorithm \ref{alg:qn} for approximating the solutions of the optimization problem \eqref{eq:opt}. To this end, we need the following ingredients.
\subsubsection{Training set}
As is customary in supervised learning (\cite{DLbook} and references therein), we need to generate or obtain data to train the network. To this end, we fix $N \in \N$ and select a set of points $\EuScript{S} = \{y_i\}_{1 \leq i \leq N}$, with each $y_i \in Y$. It is standard in machine learning that the points in the training set $\EuScript{S}$ are chosen randomly from the parameter space $Y$, independently and identically distributed with the Lebesgue measure. However, we can follow recent papers \cite{LMR1,MR1} to choose \emph{low discrepancy sequences} \cite{CAF1,owen} such as Sobol or Halton sequences as training points in order to obtain better rates of convergence of the resulting generalization error. For $Y$ in low dimensions, one can even choose (composite) Gauss quadrature points as viable training sets \cite{MM1}. 

\begin{figure}[htbp]
\centering
\includegraphics[width=8cm]{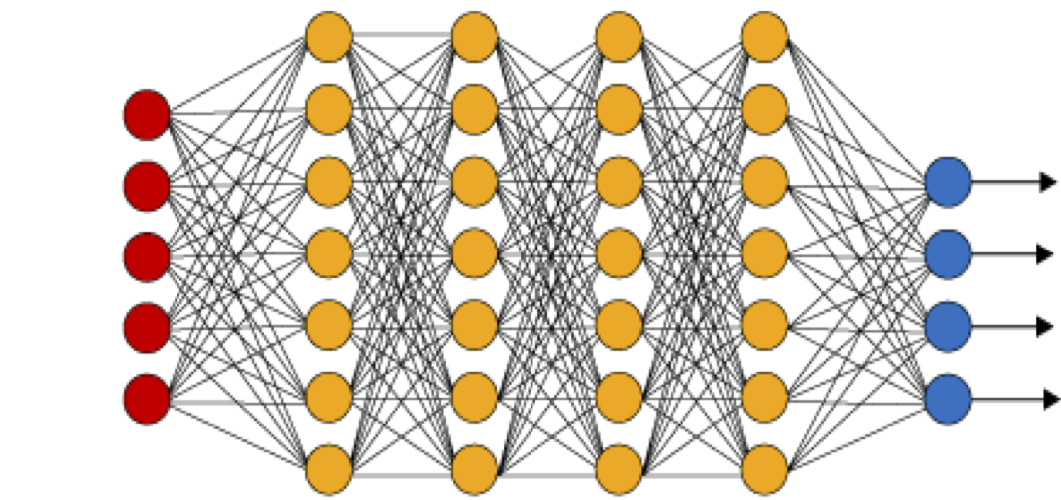}
\caption{An illustration of a (fully connected) deep neural network. The red neurons represent the inputs to the network and the blue neurons denote the output layer. They are
connected by hidden layers with yellow neurons. Each hidden unit (neuron) is connected by affine linear maps between units in different layers and then with nonlinear (scalar) activation functions within units.}
\label{fig:1}
\end{figure}

\subsubsection{Deep neural networks} 
\label{sec:NN}
Given an input vector $y \in Y$, a feedforward neural network (also termed as a multi-layer perceptron), shown in figure \ref{fig:1}, transforms it to an output through layers of units (neurons) consisting of either affine-linear maps between units (in successive layers) or scalar non-linear activation functions within units \cite{DLbook}, resulting in the representation,
\begin{equation}
\label{eq:ann1}
\map_{\theta}(y) = C_K \circ\sigma \circ C_{K-1}\ldots \ldots \ldots \circ\sigma \circ C_2 \circ \sigma \circ C_1(y).
\end{equation} 
Here, $\circ$ refers to the composition of functions and $\sigma$ is a scalar (non-linear) activation function. A large variety of activation functions have been considered in the machine learning literature \cite{DLbook}. Popular choices for the activation function $\sigma$ in \eqref{eq:ann1} include the sigmoid function, the $\tanh$ function and the \emph{ReLU} function defined by,
\begin{equation}
\label{eq:relu}
\sigma(z) = \max(z,0).
\end{equation}
When $z \in \R^p$ for some $p > 1$, then the output of the ReLU function in \eqref{eq:relu} is evaluated componentwise. 

For any $1 \leq k \leq K$, we define
\begin{equation}
\label{eq:C}
C_k (z_k) = W_k z_k + b_k, \quad {\rm for} ~ W_k \in \R^{d_{k+1} \times d_k}, z_k \in \R^{d_k}, b_k \in \R^{d_{k+1}}.
\end{equation}
For consistency of notation, we set $d_1 = d$ and $d_{K+1} = 1$. 

Thus in the terminology of machine learning (see also figure \ref{fig:1}), our neural network \eqref{eq:ann1} consists of an input layer, an output layer and $(K-1)$ hidden layers for some $1 < K \in \N$. The $k$-th hidden layer (with $d_{k+1}$ neurons) is given an input vector $z_k \in \R^{d_k}$ and transforms it first by an affine linear map $C_k$ \eqref{eq:C} and then by a ReLU (or another) nonlinear (component wise) activation $\sigma$ \eqref{eq:relu}. A straightforward addition shows that our network contains $\left(d + 1 + \sum\limits_{k=2}^{K} d_k\right)$ neurons. 
We also denote, 
\begin{equation}
\label{eq:theta}
\theta = \{W_k, b_k\}, \theta_W = \{ W_k \}\quad \forall~ 1 \leq k \leq K,
\end{equation} 
to be the concatenated set of (tunable) weights for our network. It is straightforward to check that $\theta \in \Theta \subset \R^M$ with
\begin{equation}
\label{eq:ns}
M = \sum\limits_{k=1}^{K} (d_k +1) d_{k+1}.
\end{equation}
\subsection{Loss functions and optimization} 
For any $y \in \train$, one can readily compute the output of the neural network $\map_{\theta} (y)$ for any weight vector $\theta \in \Theta$. We define the so-called training \emph{loss function} as 
\begin{equation}
\label{eq:lf1}
J (\theta) : = \sum\limits_{y \in \train} |\map(y) - \map_{\theta} (y) |^p,
\end{equation}
for some $1 \leq p < \infty$.   

The goal of the training process in machine learning is to find the weight vector $\theta \in \Theta$, for which the loss function \eqref{eq:lf1} is minimized. 

It is common in machine learning \cite{DLbook} to regularize the minimization problem for the loss function, i.e. we seek to find
\begin{equation}
\label{eq:lf2}
\theta^{\ast} = {\rm arg}\min\limits_{\theta \in \Theta} \left(J(\theta) + \lambda \reg(\theta) \right).
\end{equation}  
Here, $\reg:\Theta \to \R$ is a \emph{regularization} (penalization) term. A popular choice is to set  $\reg(\theta) = \|\theta_W\|^q_q$ for either $q=1$ (to induce sparsity) or $q=2$, with $\theta_W$ defined in \eqref{eq:theta}. The parameter $0 \leq \lambda \ll 1$ balances the regularization term with the actual loss $J$ \eqref{eq:lf1}. 

The above minimization problem amounts to finding a minimum of a possibly non-convex function over a subset of $\R^M$ for possibly very large $M$. We follow standard practice in machine learning by either (approximately) solving \eqref{eq:lf2} with a full-batch gradient descent algorithm or variants of mini-batch stochastic gradient descent (SGD) algorithms such as ADAM \cite{adam}. 

For notational simplicity, we denote the (approximate, local) minimum weight vector in \eqref{eq:lf2} as $\theta^{\ast}$ and the underlying deep neural network $\map^{\ast}= \map_{\theta^{\ast}}$ will be our neural network surrogate for the underlying map $\map$. 

The proposed algorithm for computing this neural network is summarized below.
\begin{algorithm} 
\label{alg:DL} {\bf Deep learning of parameters to observable map} 
\begin{itemize}
\item [{\bf Inputs}:] Underlying map $\map$ \eqref{eq:ptoob}. 
\item [{\bf Goal}:] Find neural network $\map_{\theta^{\ast}}$ for approximating the underlying map $\map$. 
\item [{\bf Step $1$}:] Choose the training set $\train = \{y_n\}$ for $y_n \in Y$, for all $1 \leq n \leq N$ such that the sequence $\{y_n\}$ is either randomly chosen or a low-discrepancy sequence or any other set of quadrature points in $Y$ . Evaluate $\map(y)$ for all $y \in \train$ by a suitable numerical method. 
\item [{\bf Step $2$}:] For an initial value of the weight vector $\overline{\theta} \in \Theta$, evaluate the neural network $\map_{\overline{\theta}}$ \eqref{eq:ann1}, the loss function \eqref{eq:lf2} and its gradients to initialize the
(stochastic) gradient descent algorithm.
\item [{\bf Step $3$}:] Run a stochastic gradient descent algorithm till an approximate local minimum $\theta^{\ast}$ of \eqref{eq:lf2} is reached. The map $\map^{\ast} = \map_{\theta^{\ast}}$ is the desired neural network approximating the map $\map$.
\end{itemize}
\end{algorithm}
\section{Constrained optimization with DNN surrogates}
\label{sec:3}
Next, we combine the deep neural network surrogate $\map^{\ast}$, generated with the deep learning algorithm \ref{alg:DL}, with the standard optimization algorithm \ref{alg:qn} in order to obtain the following algorithm,
\begin{algorithm}
\label{alg:dnnopt}
{\bf DNNopt: A deep learning based algorithm for PDE constrained optimization.}
\begin{itemize}
    \item [{\bf Inputs}] Parametrized PDE \eqref{eq:ppde}, observable \eqref{eq:ptoob}, cost function \eqref{eq:cost2}, training set $\train \subset Y$ (either randomly chosen or suitable quadrature points such as low discrepancy sequences), standard optimization algorithm \ref{alg:qn}. 
    \item [{\bf Goal}] Compute (approximate) minimizers for the PDE constrained optimization problem \eqref{eq:opt}. 
    \item [{\bf Step $1$}] Given training set $\train = \{ y_i \}$, with $y_i \in Y$ for $1\leq i \leq N$,  generate the deep neural network surrogate map $\map^{\ast}$ by running the deep learning algorithm \ref{alg:DL}. Set $\gl^{\ast}(y) = G(\map^{\ast}(y))$ for all $y \in Y$, with the cost function $G$ defined in \eqref{eq:cost2}.

    \item [{\bf Step $2$}] Draw $N$ random starting points $\tilde{y}_1, \ldots, \tilde{y}_N$ in $Y$. For each $1 \leq i \leq N$, run the standard optimization algorithm \ref{alg:qn} with cost function $\gl^{\ast}$ and starting values $\tilde{y}_i$ till tolerance to obtain a set $\bar{y}_i$ of approximate minimizers to the optimization problem \ref{eq:opt}.
\end{itemize}
\end{algorithm}
The following remarks about algorithm \ref{alg:dnnopt} are in order,
\begin{remark}
The quasi-Newton optimization algorithm \ref{alg:qn} requires derivatives of the cost function with respect to the input parameters $y$. This boils down to evaluation of $\nabla_y \map^{\ast}$. As $\map^{\ast}$ is a neural network of form \eqref{eq:ann1}, these derivatives can be readily and very cheaply computed with backpropagation. 
\end{remark}
\begin{remark}
\label{rem:cost}
The overall cost of DNNopt algorithm \ref{alg:dnnopt} is expected to be dominated by the cost of generating the training data $\map(y_i)$, for $y_i \in \train$, as this involves calls to the underlying and possibly expensive PDE solver. In practice, for observables $\map$, the cost of training the network can be much smaller than the cost of generating the training data, whereas the cost of evaluating the map $\map^{\ast}$ and its gradient $\nabla_y \map^{\ast}$ is negligible  (see Table 5 of the recent paper \cite{LMR1} for a comparison of these costs for a realistic example in aerodynamics). Thus, the cost of running the optimization  algorithm \ref{alg:qn} for the neural network surrogate is negligible, even for a large number of iterations.  
\end{remark}
\subsection{Analysis of the DNNopt algorithm \ref{alg:dnnopt}}
\label{sec:an1}
As discussed in remark~\ref{rem:cost}, the cost of DNNopt algorithm \ref{alg:dnnopt} is dominated by the cost of generating the training data. How many training samples $N$ are needed in order to achieve a desired accuracy for the optimization procedure ? This question is very hard to answer in full generality. However, we will investigate it in a special case, by imposing suitable hypotheses on the underlying map $\map$, on the cost function $\gl$, on the training procedure for the neural network and on the trained neural network. 

For the rest of this paper, we set $|.| = |.|_{\infty}$ which is the max vector norm. We start with suitable hypothesis on the underlying map $\map$ in \eqref{eq:opt}.
\begin{equation}
    \label{eq:h1}
    (H1) \quad \map:Y \to Z \subset \R, \quad \map\in W^{2,\infty}(Y), \quad \cost_{\map}:= \|\map\|_{W^{2,\infty}(Y)}.
\end{equation}
Next, we have the following hypothesis on the function $G$ that appears in the cost function \eqref{eq:cost2},
\begin{equation}
    \label{eq:h2}
    (H2) \quad G:Z\to X \subset \R,\quad G \in W^{2,\infty}(Z), \quad \cost_{G}:= \|G\|_{W^{2,\infty}(Z)}.
\end{equation}
Recalling that the cost function in \eqref{eq:cost2} is given by $\gl = G(\map)$, the hypotheses $H1$ and $H2$ imply that $\gl \in W^{2,\infty}(Y;X)$ with $\cost_{\gl} = \|\gl\|_{W^{2,\infty}(Y)}$ such that $\cost_{\gl} \leq \cost_{G} \cost_{\map}$. However, we need further hypothesis on this cost function given by,
\begin{equation}
    \label{eq:h34}
    \begin{aligned}
    &(H3) \quad \gl~{\rm has~exactly~one~global~minimum~at}~\bar{y}~{\rm and~no~other~local~minima}, \\
    &(H4)~\nabla \gl:Y \to \rdl \subset \R^d~{\rm is~invertible~with~inverse}~\nabla \gl^{-1}: \rdl \to Y~{\rm and}~\nabla \gl^{-1}\in W^{1,\infty}(\rdl),~\cost_{-1}:= \|\nabla \gl^{-1}\|_{W^{1,\infty}}
    \end{aligned}
\end{equation}
Note that the assumptions $H3$ and $H4$ are automatically satisfied by strictly convex functions. Thus, these hypotheses can be considered as a slight generalization of strict convexity. 

We are interested in analyzing the DNNopt algorithm \ref{alg:dnnopt}, with the underlying map and cost function satisfying the above hypothesis. Our aim is to derive estimates on the distance between the global minimizer $\bar{y}$ of the cost function $\gl$ (see hypothesis $H3$) and the approximate minimizers $\bar{y}_i$, for $1 \leq i \leq N$, generated by the DNNopt algorithm \ref{alg:dnnopt}. However, the standard training procedure for neural networks (see algorithm \ref{alg:DL}) consists of minimizing loss functions in $L^p$-norms and this may not suffice in controlling pointwise errors near extremas of the underlying function. Thus, we need to train the neural network to minimize stronger norms that can control the derivative. 

To this end, we fix $N \in \N$ and choose the training set $\train = \{y_i\}$ for $1 \leq i \leq N$ such that each $y_i \in Y$ and the training points satisfty,
\begin{equation}
    \label{eq:h5}
    (H5)\quad Y \subset \cup_{i=1}^N B\left(y_i,N^{-\frac{1}{d}}\right),
\end{equation}
with $B\left(y,r\right)$ denoting the $|.|_{\infty}$ ball, centered at $y$, with diameter $r$. 

The simplest way to enforce hypothesis $H5$ is to divide the domain $Y$ into $N$ cubes, centered at $y_i$, with diameter $N^{-\frac{1}{d}}$. 

On this training set, we assume that for any parameter $\theta \in \Theta$, the neural network $\map_{\theta} \in W^{2,\infty}(Y)$ by requiring that the activation function $\sigma \in W^{2,\infty}(\R)$ and consider the following (\emph{Lipschitz}) loss function:
\begin{equation}
    \label{eq:llf}
    \er_{N,T}^{Lip}(\theta) = \max\limits_{1 \leq i \leq N} \left|\map(y_i)-\map_{\theta}(y_i)\right| + \left|\nabla\map(y_i)-\nabla\map_{\theta}(y_i)\right|
\end{equation}
Then, a stochastic gradient descent algorithm is run in order to find 
\begin{equation}
    \label{eq:llf1}
    \theta^{\ast} = \argmin\limits_{\theta \in \Theta}  \er_{N,T}^{Lip}(\theta).
\end{equation}
Note that the gradients of the Lipschitz loss function \eqref{eq:llf} can be computed by backpropagation. We denote the trained neural network as $\map^{\ast} = \map_{\theta^{\ast}}$ and set 
\begin{equation}
    \label{eq:llf2}
    \er_{N,T}^{Lip,\ast} = \er_{N,T}^{Lip}(\theta^{\ast})
\end{equation}
Finally, we need the following assumption on the standard optimization algorithm \ref{alg:qn} i.e.,  for all starting values $y_i \in \train$, the optimization algorithm \ref{alg:qn} converges (with a finite number of iterations) to a local minimum $\bar{y}_i \in Y$ of the cost function $\gl^{\ast} = \gl(\map^{\ast})$ such that 
\begin{equation}
    \label{eq:h6}
    (H6)\quad \nabla \gl^{\ast}\left(\bar{y}_i\right) \equiv 0. 
\end{equation}
Note that we are not assuming that $\gl^{\ast}$ is convex, hence, only convergence to a local minimum can be realized. 

We have the following bound on the minimizers, generated by the DNNopt algorithm \ref{alg:dnnopt} in this framework, 
\begin{lemma}
\label{lem:1}
Let the underlying map $\map$ in \eqref{eq:opt} satisfy $H1$, the function $G$ satisfy $H2$, the cost function $\gl$ satisfy $H3$ and $H4$. Let $\map^{\ast}$ be a neural network surrogate for $\map$, generated by algorithm \ref{alg:DL}, with training set $\train = \{y_i\}$, for $1 \leq i \leq N$, such that the training points satisfy $H5$ and the neural network is obtained by minimizing the Lipschitz loss function \eqref{eq:llf}. Furthermore, for starting values $y_i \in \train$ as inputs to the standard optimization algorithm \ref{alg:qn}, let $\bar{y_i}$ denote the local minimizers of $\gl^{\ast} = G(\map^{\ast})$, generated by the algorithm \ref{alg:qn} and satisfy $H6$, then we have the following bound,
\begin{equation}
    \label{eq:dnbd}
    \begin{aligned}
    |\bar{y} - \bar{y}_j| &\leq (1+\cost_{\map})\cost_{-1}\cost_{G}\left(\er_{N,T}^{Lip,\ast} + \left(\cost_{\map} + \cost_{\map^{\ast}} \right) N^{-\frac{1}{d}}\right),
    \end{aligned}
\end{equation}
with constants $\cost_{\map,-1,G}$ are defined in \eqref{eq:h1}-\eqref{eq:h34}, $\cost_{\map^{\ast}} = \|\map^{\ast}\|_{W^{2,\infty}}$ and $\er_{N,T}^{Lip,\ast}$ is defined through \eqref{eq:llf}, \eqref{eq:llf2}. 
\end{lemma}
\begin{proof}
Fix any $1 \leq j \leq N$ and let $\bar{y}_j$ be the minimizer of $\gl^{\ast}$, generated by the standard optimization algorithm \ref{alg:qn} for starting value $\tilde{y}_j$. We have the following inequalities, 
\begin{align*}
    |\bar{y} - \bar{y}_j| &= \left|\nabla \gl^{-1}\left(\nabla \gl(\bar{y})\right) - \nabla \gl^{-1}\left(\nabla \gl(\bar{y}_j)\right)\right| \quad {\rm by}~(H4), \\
    &\leq \cost_{-1} \left|\nabla \gl(\bar{y}) - \nabla \gl(\bar{y}_j)\right|, \quad {\rm by}~(H4), \\
    &= \cost_{-1}\left|\nabla \gl(\bar{y}_j)\right|, \quad {\rm by}~(H3), \\
    &= \cost_{-1} \left|\nabla \gl(\bar{y}_j) - \nabla \gl^{\ast}(\bar{y}_j)\right|, \quad {\rm by}~(H6),
\end{align*}
Using the straightforward calculation $\nabla \gl(y) = G^{\prime}(\map(y))\nabla \map(y)$ and similarly for $\nabla \gl^{\ast}$, in the above inequality yields,
\begin{equation}
    \label{eq:l1}
\begin{aligned}
  |\bar{y} - \bar{y}_j| &\leq  \cost_{-1} \left|G^{\prime}(\map(\bar{y}_j))\nabla \map(\bar{y}_j) -    G^{\prime}(\map^{\ast}(\bar{y}_j))\nabla \map^{\ast}(\bar{y}_j)
\right|    \\
&\leq \underbrace{\cost_{-1}\left|G^{\prime}(\map(\bar{y}_j))\nabla \map(\bar{y}_j) -    G^{\prime}(\map^{\ast}(\bar{y}_j))\nabla \map(\bar{y}_j)\right|}_{T_1}  \\
&+ \underbrace{\cost_{-1}\left|G^{\prime}(\map^{\ast}(\bar{y}_j))\nabla \map(\bar{y}_j) -    G^{\prime}(\map^{\ast}(\bar{y}_j))\nabla \map^{\ast}(\bar{y}_j)\right|}_{T_2}
\end{aligned}
\end{equation}
We estimate the terms $T_{1,2}$ as follows,
\begin{equation}
    \label{eq:l2}
\begin{aligned}
    T_1 &=\cost_{-1}\left|G^{\prime}(\map(\bar{y}_j))\nabla \map(\bar{y}_j) -    G^{\prime}(\map^{\ast}(\bar{y}_j))\nabla \map(\bar{y}_j)\right| \\
    &\leq \cost_{-1}\cost_{\map}\cost_{G}|\map(\bar{y}_j) - \map^{\ast}(\bar{y}_j)|, \quad {\rm by}~H1,H2,
    \end{aligned}
    \end{equation}
and
\begin{equation}
    \label{eq:l3}
\begin{aligned}
T_2 &= \cost_{-1}\left|G^{\prime}(\map^{\ast}(\bar{y}_j))\nabla \map(\bar{y}_j) -    G^{\prime}(\map^{\ast}(\bar{y}_j))\nabla \map^{\ast}(\bar{y}_j)\right| \\
&\leq \cost_{-1}\cost_{G}|\nabla \map(\bar{y}_j) - \nabla \map^{\ast}(\bar{y}_j)|, \quad {\rm by}~H2,
\end{aligned}
\end{equation}
As $\bar{y}_j \in Y$, by the hypothesis $H5$ on the training set, there exists an $1 \leq i_j \leq N$ such that $y_{i_j} \in \train$ and $|\bar{y}_j - y_{i_j}| \leq N^{-\frac{1}{d}}$. Hence, we can further estimate $T_{1,2}$ from the following inequalities,
\begin{equation}
    \label{eq:l4}
    \begin{aligned}
    |\map(\bar{y}_j) - \map^{\ast}(\bar{y}_j)| &\leq |\map(\bar{y}_j) - \map({y}_{i_j})| + |\map^{\ast}(\bar{y}_j) - \map^{\ast}({y}_{i_j})|
    + |\map({y}_{i_j}) - \map^{\ast}({y}_{i_j})| \\
    &\leq \er_{N,T}^{Lip,\ast} + \left(\cost_{\map} + \cost_{\map^{\ast}} \right) N^{-\frac{1}{d}} \quad {\rm by}~ H5~{\rm and}~\eqref{eq:llf},\eqref{eq:llf2},
    \end{aligned}
\end{equation}
Here, we recall that $\cost_{\map^{\ast}} = \|\map^{\ast}\|_{W^{2,\infty}}$. Similarly, we have,
\begin{equation}
    \label{eq:l5}
    \begin{aligned}
    |\nabla \map(\bar{y}_j) - \nabla\map^{\ast}(\bar{y}_j)| &\leq |\nabla\map(\bar{y}_j) - \nabla\map({y}_{i_j})| + |\nabla\map^{\ast}(\bar{y}_j) - \nabla\map^{\ast}({y}_{i_j})|
    + |\nabla\map({y}_{i_j}) - \nabla\map^{\ast}({y}_{i_j})| \\
    &\leq \er_{N,T}^{Lip,\ast} + \left(\cost_{\map} + \cost_{\map^{\ast}} \right) N^{-\frac{1}{d}} \quad {\rm by}~ H5~{\rm and}~\eqref{eq:llf},\eqref{eq:llf2},
    \end{aligned}
\end{equation}
Applying \eqref{eq:l4} and \eqref{eq:l5} into \eqref{eq:l2} and \eqref{eq:l3} and then substituting the result into \eqref{eq:lf1} yields the desired inequality \eqref{eq:dnbd}
\end{proof}
In practice, we do not know the exact minimizer $\bar{y}$ of the cost function $\gl$ and hence the minimum cost $\gl(\bar{y})$. Instead, we can monitor the effectiveness of the DNNopt algorithm \ref{alg:dnnopt} by either computing the \emph{range} of minimizers i.e., 
\begin{equation}
    \label{eq:ran}
    {\rm range}(\gl):= \max\limits_{1\leq i,j \leq N} |\gl(\bar{y}_i) - \gl(\bar{y}_j)|,
\end{equation}
or an empirical approximation to the standard deviation,
\begin{equation}
    \label{eq:std}
    {\rm std}(\gl):= \sqrt{\frac{1}{N} \sum\limits_{i=1}^N |\gl(\bar{y}_i) - \bar{\gl}|^2}, \quad \bar{\gl}:= \frac{1}{N}\sum\limits_{i=1}^N \gl(\bar{y}_i).
\end{equation}
A straightforward application of the estimate \eqref{eq:dnbd}, together with \eqref{eq:h1}, \eqref{eq:h2} yields,
\begin{equation}
    \label{eq:dnbd1}
    \max\{{\rm range}(\gl),{\rm std}(\gl)\} \leq 2(\cost_{\map}+\cost_{\map}^2)\cost_{-1}\cost^2_{G}\left(\er_{N,T}^{Lip,\ast} + \left(\cost_{\map} + \cost_{\map^{\ast}} \right) N^{-\frac{1}{d}}\right) 
\end{equation}
In order to simplify notation while tracking constants, by rescaling the underlying functions, we can assume that $\cost_{\map},\cost_{G} \leq \frac{1}{3}$, then the bound \eqref{eq:dnbd} simplifies to 
\begin{equation}
    \label{eq:dnbds}
    |\bar{y} - \bar{y}_j| \leq \frac{4\cost_{-1}}{9}\left(\er_{N,T}^{Lip,\ast} + \left(\frac{1}{3}+\cost_{\map^{\ast}}\right) N^{-\frac{1}{d}}\right)
\end{equation}
Next, we define a \emph{well-trained network} as a deep neural network $\map^{\ast}$, generated by the deep learning algorithm \ref{alg:DL}, which further satisfies for any training set $\train = \{y_i\}$, with $1 \leq i \leq N$, the assumptions,
\begin{equation}
    \label{eq:wtn}
    \cost_{\map^{\ast}} = \|\map^{\ast}\|_{W^{2,\infty}} \leq \frac{1}{3}, \qquad \er_{N,T}^{Lip,\ast} \leq \frac{1}{3} N^{-\frac{1}{d}}.
\end{equation}
Hence, for a well-trained neural network, the bound \eqref{eq:dnbds} further simplifies to,
\begin{equation}
    \label{eq:dn1}
    |\bar{y} - \bar{y}_j| \leq \frac{4\cost_{-1}}{9}N^{-\frac{1}{d}}
\end{equation}
In particular, the bound \eqref{eq:dn1} ensures convergence of all the approximate minimizers $\bar{y}_i$, generated by the DNNopt algorithm \ref{alg:dnnopt}, to the underlying minimizer, $\bar{y}$, of the optimization problem \eqref{eq:opt} as the number of training samples $N \rightarrow \infty$.
\subsubsection{Discussion on the assumptions in Lemma~\ref{lem:1}}
The estimates \eqref{eq:dnbd}, \eqref{eq:dn1} need several assumptions on the underlying map, cost function and trained neural network to hold. We have the following comments about these assumptions,
\begin{itemize}
    \item The assumptions $H1$ and $H2$ entail that the underlying map $\map$ and function $G$ are sufficiently regular i.e., $W^{2,\infty}$. A large number of underlying maps for PDE constrained optimization, particularly for elliptic and parabolic PDEs satisfy these assumptions \cite{BS1}. 
    \item The assumptions $H3$ and $H4$ essentially amount to requiring that the cost function $\gl$ is strictly convex. It is standard in optimization theory to prove guarantees on the algorithms for convex cost functions, while expecting that similar estimates will hold for approximating \emph{local minima} of non-convex functions, see for instance \cite{adam}. 
    \item The hypothesis $H5$ on the training set is an assumption on how the training set \emph{covers} the underlying domain \cite{CS1}. It can certainly be relaxed from a deterministic to a random selection of points with some further notational complexity, see \cite{CALD} for space filling arguments. It also stems from an \emph{equidistribution} requirement on the training set and can be satisfied by low-discrepancy sequences such as Sobol points. 
    \item The use of Lipschtiz loss functions \eqref{eq:llf} for training the neural network is necessitated by the need to control the neural network pointwise. We can also use a \emph{Sobolev} loss function of the form,
    \begin{equation}
        \label{eq:lsob}
        \er_{N,T}^{Sob}(\theta) = \frac{1}{N}\left(\sum\limits_{i=1}^N |\map(y_i) - \map_{\theta}(y_i)|^s + \sum\limits_{i=1}^N |\nabla\map(y_i) - \nabla\map_{\theta}(y_i)|^s +\sum\limits_{i=1}^N |\nabla^2\map(y_i) - \nabla^2\map_{\theta}(y_i)|^s\right), 
    \end{equation}
with $s > d$. By Morrey's inequality, minimizing the above loss function will automatically control the differences between the (derivatives of) the underlying map and the trained neural network, pointwise and an estimate, similar to \eqref{eq:dnbd}, can be obtained. Using an $L^p$-loss function such as \eqref{eq:lf2} may be insufficient in providing control on pointwise differences between minima of the underlying cost function and approximate minima generated by the DNNopt algorithm \ref{alg:dnnopt}.     
\item The hypothesis $H6$ on the standard optimization algorithm \ref{alg:qn} is only for notational convenience and can readily be relaxed to an approximation i.e.,  $|\nabla\gl^{\ast}(\bar{y}_i)| \leq \epsilon$, for each $i$ and for a small tolerance $\epsilon \ll 1$. This merely entails a minor correction to the estimate \eqref{eq:dnbd}. Quasi-Newton algorithms are well known to converge to such approximate local minima \cite{lbfgs}.
\item The point of assuming \emph{well-trained networks} in the sense of \eqref{eq:wtn} is to simplify the estimates \eqref{eq:dnbd}, \eqref{eq:dnbd1}. The first inequality in \eqref{eq:wtn} is an assumption on uniform bounds on the trained neural networks (and their gradients) with respect to the number of training samples. The constant $1/3$ is only for notational convenience and can be replaced by any other constant. Note that we are not explicitly ensuring that the first inequality in \eqref{eq:wtn} is always satisfied. This can be readily monitored in practice and can be ensured by adding suitable regularization terms in the loss function \eqref{eq:lf2}. On the other hand, the second inequality in \eqref{eq:wtn} is an assumption on the training error which greatly serves to simplify the complexity estimates but is not necessary. The essence of this assumption is to require that training errors are smaller than the generalization gap. In general, there are no rigorous estimates on the training error with a stochastic gradient method as a very high-dimensional non-convex optimization problem is being solved.
\item In Lemma \ref{lem:1}, we have assumed that the neural network $\map^{\ast} \in W^{2,\infty}(Y)$. This rules out the use of ReLU activation function but allows using other popular activation functions such as the hyperbolic tangent. 
\end{itemize}
\section{Iterative surrogate model optimization (ISMO)}
\label{sec:4}
Even with all the hypotheses on the underlying map, cost function and trained neural network, the bound \eqref{eq:dnbd} (and the simplified bound \eqref{eq:dn1}) highlight a fundamental issue with the DNNopt algorithm \ref{alg:dnnopt}. This can be seen from the right hand side of \eqref{eq:dnbd}, \eqref{eq:dn1}, where the error with respect to approximating the minimum $\bar{y}$ of the underlying cost function in \eqref{eq:opt}, decays with a dimension dependent exponent $\frac{1}{d}$, in the number of training samples (and hence complexity of the algorithm). This exponent implies a very strong \emph{curse of dimensionality} for the DNNopt algorithm \ref{alg:dnnopt}. In particular, for optimization problems in even moderate dimensions, there could be a slow decay of the error as well as a large variance, see estimate \eqref{eq:dnbd1} on the range and standard deviation, of the computed minimizers with the DNNopt algorithm, making it very sensitive to the starting values and other hyperparameters of the underlying optimization algorithm \ref{alg:qn}. This slow rate of decay necessitates a large number of training samples, significantly increasing the cost of the algorithm and reducing its competitiveness with respect to standard optimization algorithms. 

As mentioned in the introduction, one possible cause of the slow rate of error decay in bound \eqref{eq:dnbd} lies in the fact that the training set $\train$ for the deep learning algorithm \ref{alg:DL} is chosen \emph{a priori} and does not reflect any information about the location of (local) minima of the underlying cost function in \eqref{eq:opt}. It is very likely that incorporating such information will increase the accuracy of the deep learning algorithm around minima. We propose the following iterative algorithm to incorporate such information,
\begin{algorithm}
\label{alg:ismo}
{\bf ISMO: Iterative Surrogate Model Optimization algorithm for PDE constrained optimization.}
\begin{itemize}
    \item [{\bf Inputs}] Parametrized PDE \eqref{eq:ppde}, observable \eqref{eq:ptoob}, cost function \eqref{eq:cost2}, \emph{initial} training set $\train_0 \subset Y$ with $\train_0 = \{y^0_i\}$ for $1 \leq i \leq N_0$ (either randomly chosen or suitable quadrature points such as low discrepancy sequences), standard optimization algorithm \ref{alg:qn}, hyperparameters $N_0,\Delta N$ with $\Delta N \leq N_0$
    \item [{\bf Goal}] Compute (approximate) minimizers for the PDE constrained optimization problem \eqref{eq:opt}. 
     \item [{\bf Step $k=0$}] Given initial training set $\train_0 = \{ y^0_i \}$, with $y^0_i \in Y$ for $1\leq i \leq N_0$,  generate the deep neural network surrogate map $\map_0^{\ast}$ by running the deep learning algorithm \ref{alg:DL}. Set $\gl_0^{\ast}(y) = G(\map_0^{\ast}(y))$ for all $y \in Y$, with the cost function $G$ defined in \eqref{eq:cost2}. 
     
Given \emph{batch size} $\Delta N \leq N_0$, draw $\Delta N$ random starting values $\tilde{y}^{0}_1, \ldots, \tilde{y}^0_{\Delta N}$ in $Y$. Run the standard optimization algorithm \ref{alg:qn} with $\tilde{y}^0_j$  as starting values and cost function $\gl_0^{\ast}$ till tolerance, to obtain the set $\bar{\train}_0 = \{\bar{y}^0_j\}$, for $1 \leq j \leq \Delta N$, of approximate minimizers to the optimization problem \ref{eq:opt}. Set 
\begin{equation}
    \label{eq:s1}
    \train_1 = \train_0 \cup \bar{\train}_0
\end{equation}

 \item [{\bf Step $k\geq 1$}] Given training set $\train_k$ , generate the deep neural network surrogate map $\map_k^{\ast}$ by running the deep learning algorithm \ref{alg:DL}. Set $\gl_k^{\ast}(y) = G(\map_k^{\ast}(y))$ for all $y \in Y$, with the cost function $G$ defined in \eqref{eq:cost2}. 
     
Draw $\Delta N$ random starting values $\tilde{y}^{k}_1, \ldots, \tilde{y}^k_{\Delta N}$.
     
Run the standard optimization algorithm \ref{alg:qn} with $\tilde{y}^{k}_j$, $1 \leq j \leq \Delta N$,  as starting values and with cost function $\gl_k^{\ast}$ till tolerance, to obtain the set $\bar{\train}_k = \{\bar{y}^k_j\}$, for $1 \leq j \leq \Delta N$, of approximate minimizers to the optimization problem \ref{eq:opt}. Set 
\begin{equation}
    \label{eq:sk}
    \train_{k+1} = \train_k \cup \bar{\train}_k
\end{equation}
\item If for some tolerance $\epsilon \ll 1$,
\begin{equation}
    \label{eq:scrit}
    \left|\frac{1}{\Delta N}\sum\limits_{j=1}^{\Delta N} \gl(\bar{y}^k_j) - \frac{1}{\Delta N}\sum\limits_{j=1}^{\Delta N} \gl(\bar{y}^{k-1}_j)\right| \leq \epsilon,
\end{equation}
then terminate the iteration. Else, continue to step $k+1$.
    
    \end{itemize}
    \end{algorithm}
\begin{remark}
The key difference between the ISMO algorithm \ref{alg:ismo} and DNNopt algorithm \ref{alg:dnnopt} lies in the structure of the training set. In contrast to the DNNopt algorithm, where the training set was chosen a priori, the training set in the ISMO algorithm is augmented at every iteration, by adding points that are identified as approximate local minimizers (by the standard optimization algorithm \ref{alg:qn}) of the cost function $\gl^{\ast}_k$, corresponding to the surrogate neural network model $\map_k^{\ast}$. Thus, a sequence of neural network surrogates are generated, with each successive iteration likely producing a DNN surrogate with greater accuracy around local minima of the underlying cost function. Given this feedback between neural network training and optimization of the underlying cost function, we consider ISMO as an example of an \emph{active learning} algorithm \cite{AL}, with the deep learning algorithm \ref{alg:DL}, playing the role of the learner, querying the standard optimization algorithm \ref{alg:qn}, which serves as the \emph{teacher} or \emph{oracle}, for obtaining further training data.
\end{remark}
\subsection{Analysis of the ISMO algorithm \ref{alg:ismo}}
\label{sec:an2}
Does the ISMO algorithm \ref{alg:ismo} lead to more accuracy than the DNNopt algorithm \ref{alg:dnnopt}? We will investigate this question within the framework used for the analysis of the DNNopt algorithm in section \ref{sec:an1}. We recall that a modified Lipschitz loss function \eqref{eq:llf} was used for training the underlying deep neural network in the supervised learning algorithm \ref{alg:DL}. We will continue to use this Lipschitz loss function. Hence, at any step $k$ in algorithm \ref{alg:ismo}, with training set $\train_k$, the following loss function is used,
\begin{equation}
    \label{eq:llfk}
    \er_{T}^{Lip,k}(\theta) = \max\limits_{y_i \in \train_k} \left|\map(y_i)-\map_{\theta}(y_i)\right| + \left|\nabla\map(y_i)-\nabla\map_{\theta}(y_i)\right|
\end{equation}
Then, a stochastic gradient descent algorithm is run in order to find 
\begin{equation}
    \label{eq:llfk1}
    \theta^{\ast}_k = \argmin\limits_{\theta \in \Theta}  \er_{T}^{Lip,k}(\theta).
\end{equation}
We denote the trained neural network as $\map^{\ast}_k = \map_{\theta^{\ast}_k}$ and set 
\begin{equation}
    \label{eq:llfk2}
    \er_{T}^{Lip,k,\ast} = \er_{T}^{Lip}(\theta^{\ast}_k)
\end{equation}
We have the following bound on the minimizers, computed with the ISMO algorithm \ref{alg:ismo},
\begin{lemma}
\label{lem:2}
Let the underlying map $\map$ in \eqref{eq:opt} satisfy $H1$ with $\cost_{\map} \leq \frac{1}{3}$, the function $G$ satisfy $H2$ with constant $\cost_{G} \leq \frac{1}{3}$, the cost function $\gl$ satisfy $H3$ and $H4$. Let the neural networks at every step $k$ in algorithm \ref{alg:ismo} be trained with the Lipschtiz loss functions \eqref{eq:llfk} and local minimizers of the cost function $\gl_k^{\ast}$, computed with the standard optimization algorithm \ref{alg:qn} satisfy $H6$. Let the trained neural network $\map_0^{\ast}$ be well-trained in the sense of \eqref{eq:wtn} and the number of initial training points, $N_0 = \#(\train_0)$ be chosen such that
\begin{equation}
    \label{eq:l21}
    N_0 \geq \left(\frac{\cost}{\sigma_0}\right)^d, \quad {\rm for~some}~ \sigma_0 < 1,
\end{equation}
with constant $\cost = \frac{4 \cost_{-1}}{9}$. 

Given $\sigma_0$ in \eqref{eq:l21}, define $\sigma_k$ for $k \geq 1$ by,
\begin{equation}
    \label{eq:sk}
    \sigma_k = \er_T^{Lip,k,\ast} + \left(\frac{1}{3} + \cost_{\map_k^{\ast}} \right) \frac{\sigma_{k-1}}{(\Delta N)^{\frac{1}{d}}}.
\end{equation}
Under the assumption that the local minimizers $\bar{y}^k_j$, for $1 \leq j \leq \Delta N$ and at every $k \geq 0$ also satisfy,
\begin{equation}
    \label{eq:h7}
(H7)    \quad \bar{y}^k_j \in \cup_{i=1}^{\Delta N} B\left(\bar{y}^{k-1}_i,\frac{\sigma_{k-1}}{(\Delta N)^{\frac{1}{d}}} \right),
\end{equation}
then the approximate minimizers $\bar{y}^k_j$, for all $k \geq 1$, $1 \leq j \leq \Delta N$ satisfy the bound,
\begin{equation}
    \label{eq:isbd1}
    |\bar{y} - \bar{y}^k_j| \leq \cost \sigma_k,
\end{equation}
with $\sigma_k$ defined in \eqref{eq:sk}. 

In particular, if at every step $k$, the trained neural network $\map^{\ast}_k$ is \emph{well-trained} i.e.,  it satisfies,
\begin{equation}
    \label{eq:wtnk}
    \cost_{\map^{\ast}} = \|\map^{\ast}_k\|_{W^{2,\infty}} \leq \frac{1}{3}, \qquad \er_{N,T}^{Lip,k,\ast} \leq \frac{1}{3}\frac{\sigma_{k-1}}{(\Delta N)^{\frac{1}{d}}},
\end{equation}
then, the minimizers, generated by the ISMO algorithm \ref{alg:ismo} satisfy the bound,
\begin{equation}
    \label{eq:isbd}
    |\bar{y} - \bar{y}^k_j| \leq \frac{\cost^k\sigma_0}{(\Delta N)^{\frac{k}{d}}}, \quad \forall k \geq 1.
\end{equation}
\end{lemma}
\begin{proof}
We observe that step $k=0$ of the ISMO algorithm \ref{alg:ismo} is identical to the DNNopt algorithm \ref{alg:dnnopt}, with training set $\train_0$ and $N_0$ training samples.  As we have required that all the hypothesis, $H1-H6$ for Lemma~\ref{lem:1} holds and the neural network $\map^{\ast}_0$ is well-trained in the sense of \eqref{eq:wtn}, we can directly apply estimate \eqref{eq:dn1} to obtain that
\begin{equation}
\label{eq:l22}
|\bar{y} - \bar{y}^0_j| \leq \cost N_0^{-\frac{1}{d}} \leq \sigma_0 < 1, \quad {\rm by}~ \eqref{eq:l21}.
\end{equation}
We will prove the bound \eqref{eq:isbd1} by induction. Assume that \eqref{eq:isbd1} holds for step $k-1$. Given the training set $\train_k$, we obtain the trained neural network $\map^{\ast}_k$ by minimizing the Lipschitz loss function \eqref{eq:llfk} and with starting values $\tilde{y}^{k-1}_j$, for $1 \leq j \leq \Delta N$, we run the standard optimization algorithm \ref{alg:qn} to generate approximate minimizers $\bar{y}^k_j$. We proceed as in the proof of Lemma~\ref{lem:1} to calculate the following,
\begin{align*}
    |\bar{y} - \bar{y}^k_j| &= \left|\nabla \gl^{-1}\left(\nabla \gl(\bar{y})\right) - \nabla \gl^{-1}\left(\nabla \gl(\bar{y}^k_j)\right)\right| \quad {\rm by}~(H4), \\
    &\leq \cost_{-1} \left|\nabla \gl(\bar{y}) - \nabla \gl(\bar{y}^k_j)\right|, \quad {\rm by}~(H4), \\
    &= \cost_{-1}\left|\nabla \gl(\bar{y}^k_j)\right|, \quad {\rm by}~(H3), \\
    &= \cost_{-1} \left|\nabla \gl(\bar{y}^k_j) - \nabla \gl_k^{\ast}(\bar{y}^k_j)\right|, \quad {\rm by}~(H6), \\
&=  \cost_{-1} \left|G^{\prime}(\map(\bar{y}^k_j))\nabla \map(\bar{y}^k_j) -    G^{\prime}(\map_k^{\ast}(\bar{y}^k_j))\nabla \map_k^{\ast}(\bar{y}^k_j)
\right|    \\
&\leq \underbrace{\cost_{-1}\left|G^{\prime}(\map(\bar{y}^k_j))\nabla \map(\bar{y}^k_j) -    G^{\prime}(\map^{\ast}_k(\bar{y}^k_j))\nabla \map(\bar{y}^k_j)\right|}_{T_3}  
+ \underbrace{\cost_{-1}\left|G^{\prime}(\map^{\ast}_k(\bar{y}^k_j))\nabla \map(\bar{y}_j) -    G^{\prime}(\map^{\ast}_k(\bar{y}^k_j))\nabla \map_k^{\ast}(\bar{y}^k_j)\right|}_{T_4}
\end{align*}
As in the proof of Lemma~\ref{lem:1}, we can estimate,
\begin{align*}
     T_3 &\leq \cost_{-1}\cost_{\map}\cost_{G}|\map(\bar{y}^k_j) - \map^{\ast}_k(\bar{y}^k_j)| \leq \frac{\cost_{-1}}{9}|\map(\bar{y}^k_j) - \map^{\ast}_k(\bar{y}^k_j)| , \quad {\rm by}~H1,H2, \\
     T_4 & \leq \cost_{-1}\cost_{G}|\nabla \map(\bar{y}^k_j) - \nabla \map^{\ast}_k(\bar{y}^k_j)| \leq \frac{\cost_{-1}}{3}|\map(\bar{y}^k_j) - \map^{\ast}_k(\bar{y}^k_j)|, \quad {\rm by}~H2.
\end{align*}
By the hypothesis $H7$ on the approximate minimizers, there exists an $1 \leq i_{j,k} \leq \Delta N$ such that $\bar{y}^{k-1}_{i_{j,k}} \in \train_k$ and $|\bar{y}^k_j - \bar{y}^{k-1}_{i_{j,k}}| \leq \frac{\sigma_{k-1}}{(\Delta N)^{\frac{1}{d}}}$. Hence, we can further estimate from the Lipschitz loss function \eqref{eq:llfk}, completely analogously to the estimates \eqref{eq:l4} and \eqref{eq:l5} to obtain that,
\begin{align*}
    |\bar{y} - \bar{y}^k_j| &\leq \cost\left(\er_T^{Lip,k,\ast} + \left(\frac{1}{3} + \cost_{\map_k^{\ast}} \right) \frac{\sigma_{k-1}}{(\Delta N)^{\frac{1}{d}}}\right),
\end{align*}
which with definition \eqref{eq:sk} is the desired inequality \eqref{eq:isbd1} for step $k$.

A straightforward application of the definitions \eqref{eq:wtnk} for well-trained networks $\map^{\ast}_k$ for $k \geq 1$, together with a recursion on the inequality \eqref{eq:isbd1} yield the bound \eqref{eq:isbd}. 
\end{proof}
\begin{remark}
In addition to the assumptions $H1-H6$ in section \ref{sec:an1}, we also need another assumption $H7$ for obtaining bounds on the minimizers, generated by the ISMO algorithm. This assumption amounts to requiring that during the iterations, the optimizers of \eqref{eq:llfk}, at a given step $k$, remain close to optimizers of \eqref{eq:llfk}, at the previous step $k-1$. This seems to be a \emph{continuity} requirement on the minimizers. It is justified as long as the first step $k=0$ is able to provide a good approximation to the underlying (global) minimum and successive iterations improve the approximation of this global minimum. On the other hand, for non-convex cost functions, it could happen that such an assumption is violated near local minima. However, given the fact that the starting values for the standard optimization algorithm, at any iteration step $k \geq 1$ of the ISMO algorithm \ref{alg:ismo}, are chosen randomly, we would expect that this provides a pathway for the algorithm to escape from unfavorable local minima.  
\end{remark}
\begin{remark}
The assumptions on \emph{well-trained networks} \eqref{eq:wtnk}, at each step of the ISMO algorithm \ref{alg:ismo} is for conceptual and notational simplicity. An estimate, analogous to \eqref{eq:isbd}, can be proved without resorting to assuming \eqref{eq:wtnk}.   

\end{remark}
\subsection{Comparison between DNNopt and ISMO}
The estimates \eqref{eq:dn1} and \eqref{eq:isbd} provide a quantitative comparison of accuracy between the DNNopt algorithm \ref{alg:dnnopt} and the \emph{active learning} ISMO algorithm \ref{alg:ismo}.

To realize this, we further assume that the ISMO algorithm converges to the desired tolerance within $K$ iterations and that the \emph{batch size} $\Delta N = c N_0$, for some fraction $c \leq 1$. Hence, the accuracy of the ISMO algorithm is quantified by \eqref{eq:isbd} as,
\begin{equation}
    \label{eq:isk}
    |\bar{y} - \bar{y}_j^K| \leq \frac{\cost^K\sigma_0}{c^{\frac{K}{d}} N_0^{\frac{K}{d}}}.
\end{equation}
As we have assumed that the cost of both DNNopt and ISMO algorithms is dominated by the cost of generating the training samples, the cost of the ISMO algorithm with $K$ iterations is determined by the total number of training samples i.e.,  $(1+cK)N_0$. For exactly the same number of training samples, the accuracy of the DNNopt algorithm \ref{alg:dnnopt} is given by the estimate \eqref{eq:dn1} as,
\begin{equation}
    \label{eq:dnk}
    |\bar{y} - \bar{y}_j| \leq \frac{\cost}{\left(1+cK\right)^{\frac{1}{d}}N_0^{\frac{1}{d}}}.
\end{equation}
As (fixed) $N_0 \gg 1$ and $c \approx 1$, as long as $\cost \sim \Ord(1)$, we see that the estimate \eqref{eq:isk} suggests an \emph{exponential decay} of the approximation error of the underlying minimum wrt the size of the training set for the ISMO algorithm, in contrast to the \emph{algebraic decay} for the DNNopt algorithm, at-least within the framework of Lemmas \ref{lem:1} and \ref{lem:2}. 

This contrast can also be directly translated to the \emph{range} \eqref{eq:ran} and \emph{standard deviation} \eqref{eq:std} of the approximate minimizers for both algorithms. Straightforward calculations with the definitions \eqref{eq:ran}, \eqref{eq:std} and estimates \eqref{eq:isk}, \eqref{eq:dnk} reveal an exponential decay for the ISMO algorithm of the form
\begin{equation}
    \label{eq:isk1}
    \max\{{\rm range}(\gl),{\rm std}(\gl)\} \sim \Ord\left(N_0^{-\frac{K}{d}}\right),
\end{equation}
whereas there is an algebraic decay for the DNNopt algorithm of the form,
\begin{equation}
    \label{eq:dnk1}
    \max\{{\rm range}(\gl),{\rm std}(\gl)\} \sim \Ord\left((KN_0)^{-\frac{1}{d}}\right).
\end{equation}
These estimates \eqref{eq:isk1} and \eqref{eq:dnk1} suggest that the ISMO algorithm may have significantly lower variance (sensitivity) of approximate minimizers of the underlying optimization problem \eqref{eq:opt}, when compared to the DNNopt algorithm. In particular, these estimates highlight the role of the iterated feedback between the learner (the deep learning algorithm \ref{alg:DL}) and the oracle (the standard optimization algorithm \ref{alg:qn}) in providing the training data, that exponentially improves the approximation of the underlying minima, in each successive step. 

\begin{figure}[h!]
\centering
\includegraphics[width=0.98\textwidth]{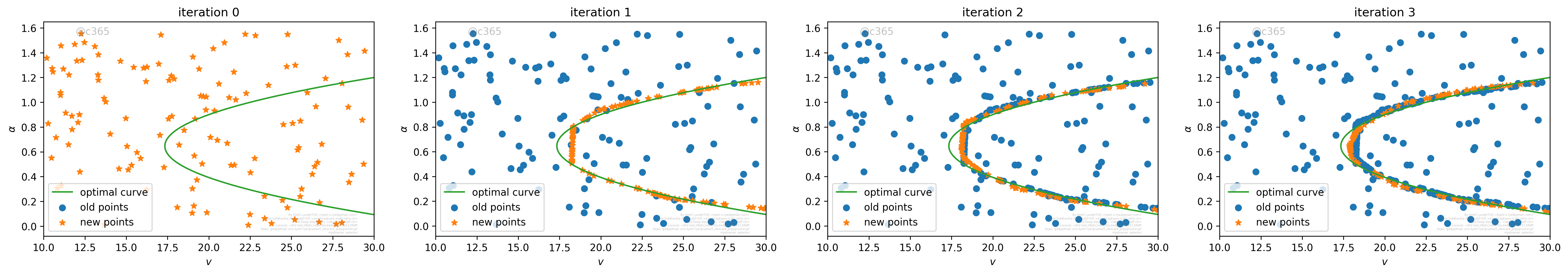}
\caption{Evolution of training sets $\train_k$, for different iterations $k$, for the ISMO algorithm \ref{alg:ismo} for the optimal control of projectile motion example. The newly added  training points $\bar{\train}_k$ are the orange dots, while existing training points $\train_k$ are blue dots. The exact minimizers are placed on the green parabola. Note how most of the newly added training points lie on this parabola.}
\label{fig:proj1}
\end{figure}    
    
\section{Numerical experiments}
\label{sec:5}
The theory presented in sections \ref{sec:3} and \ref{sec:4} suggests that the approximate minimizers computed by the DNNopt algorithm \ref{alg:dnnopt} and the ISMO algorithm \ref{alg:ismo}, converge to (local) minima of the underlying optimization problem \ref{eq:opt}, as the number of training samples is increased. Moreover, the theory also suggests that the ISMO algorithm converges significantly faster than the DNNopt algorithm and can have significantly lower variance (sensitivity) of the computed minima. However, these theoretical results were obtained under rather stringent regularity hypotheses on the underlying observable $\map$, convexity hypothesis on the cost function $\gl$, use of Lipschitz loss functions \eqref{eq:llf}, \eqref{eq:llfk}, and other hypotheses on the distribution of optimizers \eqref{eq:h7}. Many of these hypotheses may not be valid  for PDE constrained optimization problems of practical interest. Will the DNNopt and ISMO algorithms, work as predicted by the proposed theory, for these realistic problems ? In particular, will ISMO significantly outperform DNNopt in terms of accuracy and robustness ? We investigate these questions by a series of numerical experiments in this section.

\subsection{Details of implementation}
The implementation of both the DNNopt algorithm \ref{alg:dnnopt} and the ISMO algorithm \ref{alg:ismo}, is performed in Python utilizing the Tensorflow framework~\cite{tensorflow2015-whitepaper} with Keras~\cite{chollet2015keras} for the machine learning specific parts, while SciPy~\cite{2020SciPy-NMeth} is used for handling optimization of the objective function. The code is open-source, and can be downloaded from \url{https://github.com/kjetil-lye/iterative_surrogate_optimization}. New problem statements can be readily implemented without intimate knowledge of the algorithm.

The architecture and hyperparameters for the underlying neural networks in both algorithms are specified in Table \ref{tab:parameters}. To demonstrate the robustness of the proposed algorithms, we train an ensemble of neural networks with 10 different starting values and evaluate the resulting statistics. 

By utilizing the job chaining capability of the IBM Spectrum LSF system~\cite{IBM-Spectrum-LSF}, we are able to create a process-based system where each iteration of the ISMO algorithm is run as a separate job. The novel use of job chaining enables us to differentiate the job parameters (number of compute nodes, runtime, memory, etc.) for each step of the algorithm, without wasting computational resources.

\begin{table}[h]
    \centering
    \begin{tabular}{ccccccc}
        \toprule
         Width & Depth & Regulariser & Optimizer & Loss function & Learning rate & Activation function\\
         \midrule
         20 & 8 & L$^2$: $7.8\cdot 10^{-7}$ & Adam & MSE & 0.001& ReLU\\
         \bottomrule
    \end{tabular}
    \caption{Training parameters for the experiments done in section \ref{sec:5}.}
    \label{tab:parameters}
\end{table}
\subsection{Optimal control of projectile motion}

\begin{table}[h]
    \centering
    \begin{tabular}{ccccccc}
        \toprule
         $\mathbf{x}_0$ &  $g$ & $C_D$ & $m$ & $r$ & $\rho$ & $\Delta t$\\
         \midrule
         $(0.2, 0.5)$ & $9.81$ & $0.1$ & $0.142$ & $0.22$ & $1.1455$ & $0.01$\\
         \bottomrule
    \end{tabular}
    \caption{Parameters for the projectile motion experiment in section \ref{sec:5}.}
    \label{tab:projectile_motion_parameters}
\end{table}
This problem was proposed in a recent paper \cite{LMM1} as a prototype for using deep neural networks to learn observables for differential equations and dynamical systems. The motion of a projectile, subjected to gravity and air drag, is described by the following system of ODEs,
\begin{alignat}{2}
\label{eq:proj}
\frac{d}{dt} \mathbf{x}(t;y) &= \mathbf{v}(t;y), && \mathbf{x}(0;y)=\mathbf{x}_0(y), \\
\frac{d}{dt} \mathbf{v}(t;y) &= -F_D(\mathbf{v}(t;y);y)\frac{\mathbf{v}}{|\mathbf{v}|_2} - g\mathbf{e}_2, \qquad && \mathbf{v}(0;y) = \mathbf{v}_0(y),
\end{alignat}
where $F_D = \frac{1}{2m}\rho C_d \pi r^2 \|v\|^2$ denotes the drag force, the initial velocity is set to $\mathbf{v}(0;y) =[v_0 \cos(\alpha), v_0 \sin(\alpha)]$ and the  various parameters are listed in Table \ref{tab:projectile_motion_parameters}.

The observable $\map$ of interest is the \emph{horizontal range} $x_{\max}$ of the simulation:
\begin{equation*}
    x_{\max}(y) = x_1(y,t_f), \qquad \text{with } t_f = x_2^{-1}(0).
\end{equation*}
Here, $y = [\alpha,v_0]$ denotes the vector of \emph{control parameters}. The aim of this \emph{model problem} is to find initial release angle $\alpha \in [0, \pi/2]$ and initial velocity $v_0 \in [10, 30]$ such that the horizontal range $x_{max}$ is exactly $x_{ref} = 15$. Thus, we have an optimal control problem, constrained by the ODE \eqref{eq:proj}, with control parameters $y \in [0,1]^2$. The underlying objective function is given by,
\begin{equation}
    \label{eq:objproj}
    \gl(y) = \frac{1}{2}|x_{max}(y) - x_{ref}|^2.
\end{equation}
The advantage of this model problem is that the optimal control parameters, i.e.,  minimizers of \eqref{eq:objproj}, can be easily computed through cheap numerical experiments, 
which corresponds to points, lying on a parabola-like curve in the $Y$-space, shown in figure \ref{fig:proj1} (Left). Moreover, one can cheaply generate training data by solving the ODE system \eqref{eq:proj} with a standard forward Euler time discretization. 
\begin{figure}[htbp]
\centering
\includegraphics[width=8cm]{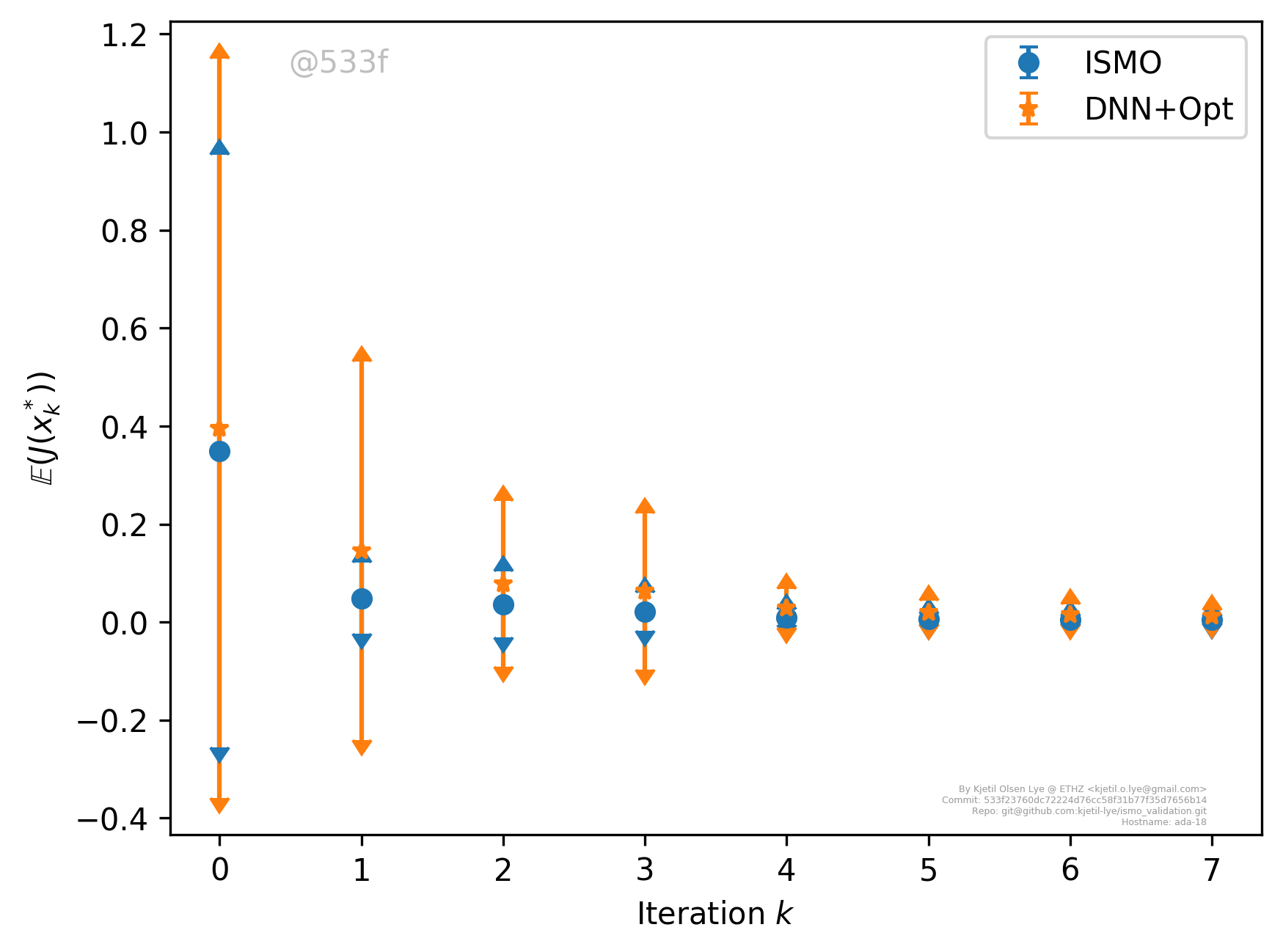}
\caption{The objective function vs. number of iterations  of the DNNopt and ISMO algorithms for the optimal control of projectile motion. The mean of the objective function and mean $\pm$ standard deviation \eqref{eq:std} are shown.}
\label{fig:proj2}
\end{figure}

We run the DNNopt algorithm \ref{alg:dnnopt} on this problem with hyperparameters  listed in Table~\ref{tab:parameters} and plot the mean $\bar{\gl}$ \eqref{eq:std} of the cost function \eqref{eq:objproj}, as well as the mean $\pm$ standard deviation \eqref{eq:std} (as a measure of the robustness/sensitivity) of the algorithm in figure \ref{fig:proj2}. For this figure, at iteration $k$, the size of the training set $\train$ in algorithm \ref{alg:dnnopt} is given by $N_k = N_0 + k\Delta N$, with $N_0 = 32$ and $\Delta N = 16$. From figure \ref{fig:proj2}, we see that the mean of the cost function, computed by the DNNopt algorithm converges to the underlying minimum $\gl = 0$, with increasing number of training samples. However, the standard deviation \eqref{eq:std} decays at a slow rate. In particular, the standard deviation is still high, for instance when $k=3$ or $k=4$. This slow decay is consistent with an algebraic rate of convergence (with respect to the number of iterations) as predicted by theory in \eqref{eq:dnbd1}, \eqref{eq:dnk1} implying that the optimizers, computed by the DNNopt algorithm \ref{alg:dnnopt} may not be robust. Note that the mean $-$ standard deviation for the objective function, might be negative on account of non-zero standard deviation near the minimum (zero). 
\begin{figure}[h!]
    \begin{subfigure}{.48\textwidth}
        \centering
        \includegraphics[width=1\linewidth]{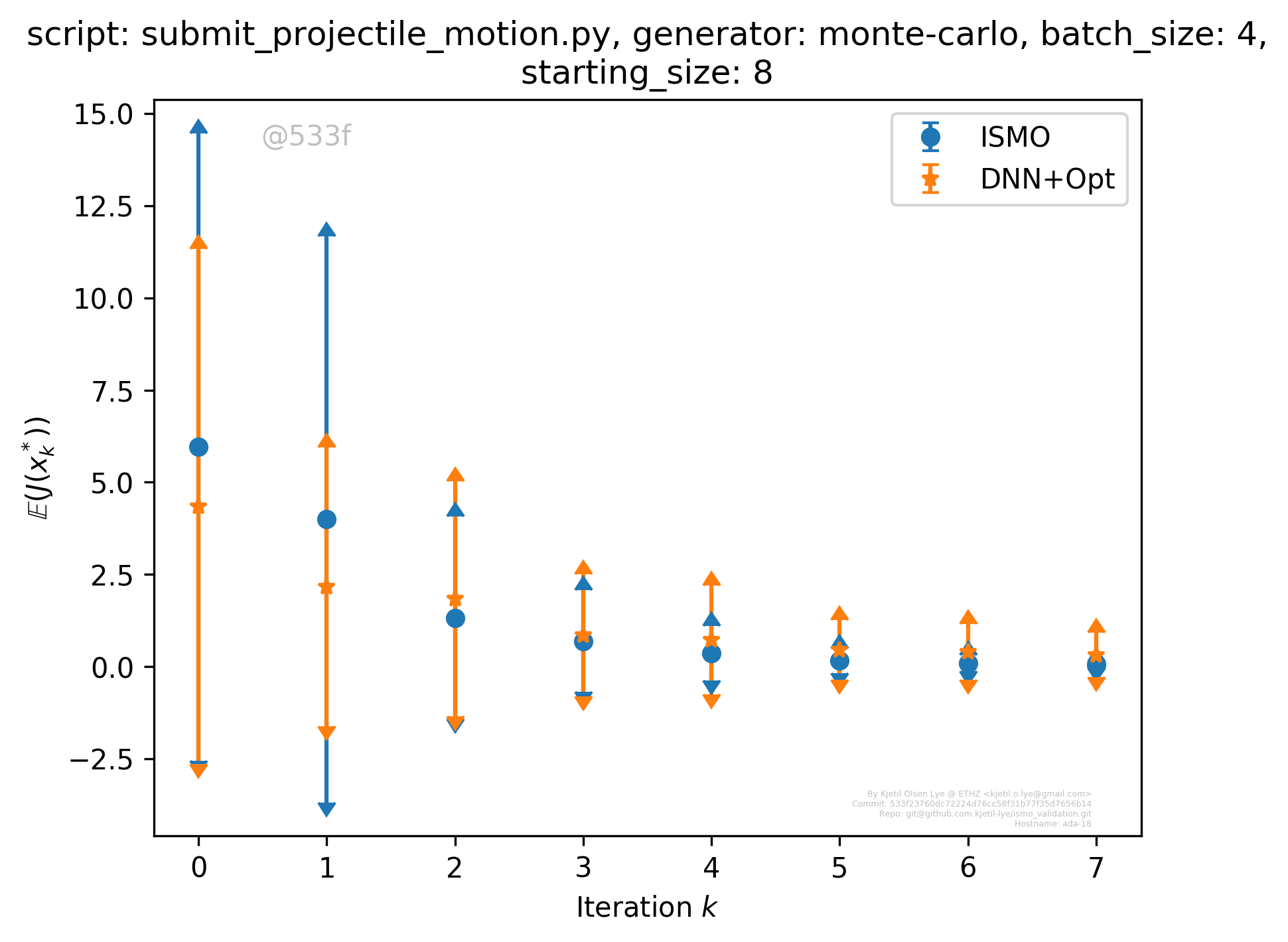}
        %\caption{Reference shape}
    \end{subfigure}
    \begin{subfigure}{.48\textwidth}
        \centering\
        \includegraphics[width=1\linewidth]{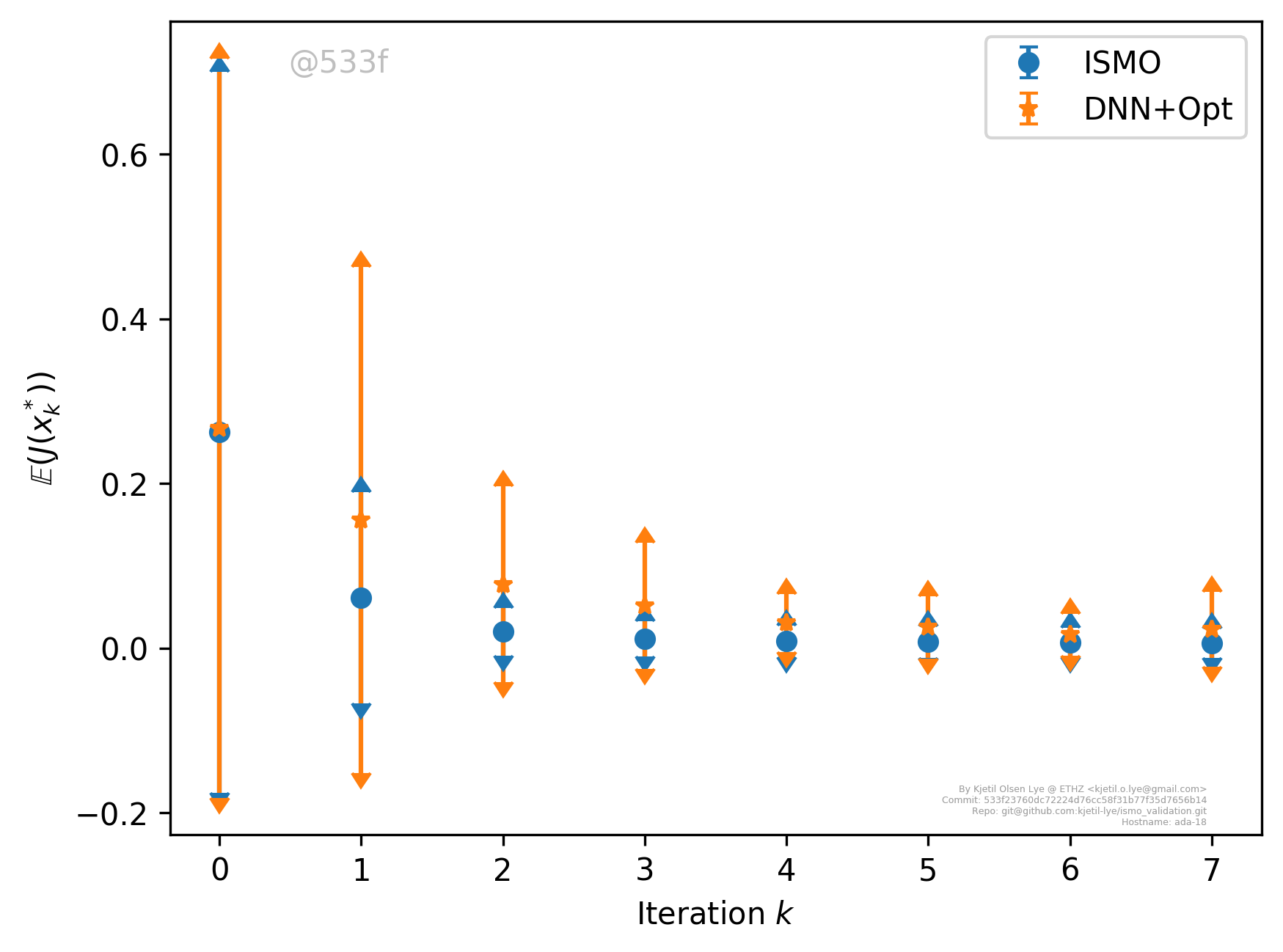}
       % \caption{Mach number (Sample)}
    \end{subfigure}
    \caption{The objective function vs. number of iterations  of the DNNopt and ISMO algorithms for the optimal control of projectile motion. The mean of the objective function and mean $\pm$ standard deviation \eqref{eq:std} are shown. Left: Random initial training points with $N_0=8$ and $\Delta N =4$ for the ISMO algorithm. Right: Sobol initial training points for the ISMO algorithm with $N_0=32$ and $\Delta N = 16$.}
    \label{fig:proj3}
    \end{figure}

We also run the ISMO algorithm \ref{alg:ismo} for this problem, with initial $N_0=32$ training samples and a batch size of $\Delta N = 16$ in algorithm \ref{alg:ismo}. The resulting mean (and mean $\pm$ standard deviation) of the cost function are plotted in figure \ref{fig:proj2}. From this figure, we see that the mean of the cost function quickly converges to $0$. Moreover, and in contrast to the DNNopt algorithm, the standard deviation with ISMO, converges very rapidly to zero, indicating consistency with the exponential convergence, predicted by the theory in  \eqref{eq:isbd}, \eqref{eq:isk}. This example also brings out a key advantage with ISMO, namely although the mean of the objective function is only slightly lower with ISMO than DNNopt in this example, the standard deviation is much lower. Hence, ISMO, at the same cost, is significantly more robust, than the DNNopt algorithm. 

Why is ISMO more robust than DNNopt in this case ?. As explained in the theory for ISMO, the reason lies in the iterative feedback between the learner (supervised learning algorithm \ref{alg:DL}) that queries the teacher (standard optimization algorithm \ref{alg:qn}) for augmenting training points, that are more clustered around the underlying minima, in each successive iteration. This is illustrated in figure \ref{fig:proj1}, where we plot the training set $\train_k$ and the additional training points $\bar{\train}_k$ at different iterations $k$ for the ISMO algorithm. As observed in this figure, we start with a randomly chosen set of points, which have no information on the minimizers of the underlying function (the parabola in figure \ref{fig:proj1}). At the very first iteration, the standard optimization algorithm identifies optimizers that are clustered near the parabola and at each successive step, more such training points are identified, which increasingly improve the accuracy and robustness of the ISMO algorithm. 

We recall from the theory presented in Lemma~\ref{lem:2} that the initial number of training points $N_0$ in the ISMO algorithm should be sufficiently high in order to enforce \eqref{eq:l21} and obtain the desired rate of convergence. To check whether a minimum number of initial training points are necessary for the ISMO algorithm to be robust and accurate, we set $N_0 = 8$ and $\Delta N = 4$ in algorithm \ref{alg:ismo} and present the results in figure \ref{fig:proj3} (left). As seen from this figure and contrasting the results with those presented in figure \ref{fig:proj2}, we observe that the ISMO algorithm is actually \emph{less robust} than the DNNopt algorithm, at least for the first few iterations. Hence, we clearly need enough initial training samples for the ISMO algorithm to be robust and efficient. 

Finally, the results presented so far used randomly chosen initial training points for the ISMO algorithm. The results with low-discrepancy Sobol points are presented in figure \ref{fig:proj3} (right) and show the same qualitative behaviour as those with randomly chosen points. Although the mean and standard deviation of the objective function are slightly smaller with the Sobol points, the differences are minor and we will only use randomly chosen initial training points for the rest of the paper. 
\begin{figure}[h!]
\centering
\includegraphics[width=8cm]{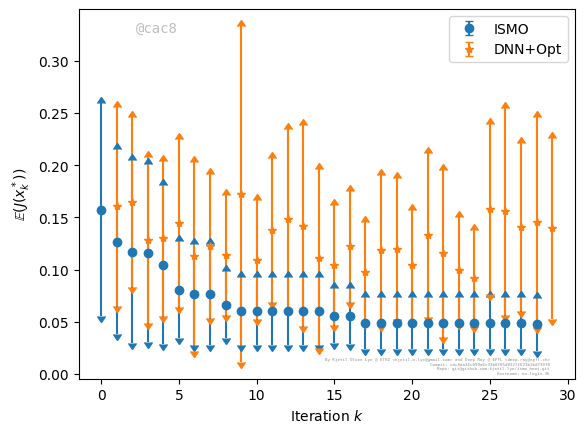}
\caption{The objective function \eqref{eq:htct} vs. number of iterations  of the DNNopt and ISMO algorithms for the parameter identification (data assimilation) problem for the heat equation. The mean of the objective function and mean $\pm$ standard deviation \eqref{eq:std} are shown.}
\label{fig:heat1}
\end{figure}

This numerical experiment is openly avaiable at \url{https://github.com/kjetil-lye/ismo_validation}.
\subsection{Inverse problem for the heat equation}
In this section, we consider our first PDE example, the well-known heat equation in one space dimension, given by,
\begin{equation}
    \label{eq:ht}
    \begin{aligned}
    u_t &= qu_{xx}, \quad (x,t) \in (0,1)\times(0,T), \\
    u(x,0) &= \bar{u}(x), \quad x \in (0,1), \\
    u(0,t) &= u(1,t) \equiv 0, \quad t \in (0,T). 
    \end{aligned}
\end{equation}
Here, $u$ is the temperature and $q \ge 0$ is the diffusion coefficient. We will follow \cite{BV1} and identify parameters for the following \emph{data assimilation problem}, i.e.,  for an initial data of the form,
\begin{equation}
    \label{eq:htin}
    \bar{u}(x) = \sum\limits_{\ell=1}^L a_{\ell}\sin(\ell \pi x),
\end{equation}
with unknown coefficients $a_{\ell}$ and for an unknown diffusion coefficient $q$, the aim of this data assimilation or parameter identification problem is to compute approximations to $a_\ell,q$ from measurements of the form $u_i = u(x_i,T)$, i.e., pointwise measurements of the solution field at the final time. 

This problem can readily be cast in the framework of PDE constrained optimization, section \ref{sec:2}, in the following manner. The control parameters are $y = [q,a_1,a_2,\ldots,a_L] \in Y = [0,1]^{L+1}$ and the relevant observables are,
\begin{equation}
    \label{eq:htobs}
    \map_i(y) = u(x_i,T), \quad 1 \leq i \leq I.
\end{equation}
Following \cite{BV1} and references therein, we can set up the following cost function,
\begin{equation}
    \label{eq:htct}
    \gl(y):= \frac{1}{2I} \sum\limits_{i=1}^I |\map_i(y) - \bar{\map}_i |^2,
\end{equation}
with measurements $\bar{\map}_i$. In this article, we assume that the measurements are noise free. 

In order to perform this experiment, we set $L = 9$, $I=5$ and $T=0.001$ and generate the measurement data from the following \emph{ground truth} parameters,
$$
[q,a_1,\ldots a_9] = [0.75,0.1,0.4,0.8,0.25,0.75,0.45,0.9,0.25,0.85],
$$
and measurement points,
$$
[x_1,\ldots,x_5] = [0.125, 0.25, 0.5, 0.75, 0.825]
$$
\begin{figure}[htbp]
\centering
\includegraphics[width=8cm]{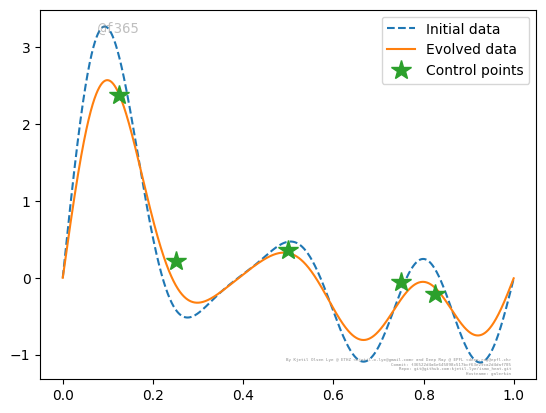}
\caption{Solution of the data assimilation problem for the heat equation with the ISMO algorithm. The identified initial data \eqref{eq:htin}, resulting (exact) solution at final time and measurements \eqref{eq:htobs} are shown. }
\label{fig:heat2}
\end{figure}
We numerically approximate the heat equation \eqref{eq:ht} with a second-order Crank-Nicholson finite difference scheme with $2048$ points in both space and time and generate the measurement data $\bar{\map}_i$, for $1 \leq i \leq 5$ from this simulation. 

The DNNopt algorithm \ref{alg:dnnopt} is run for this particular configuration with the hyperparameters in Table~\ref{tab:parameters}. We would like to point out that $I$ deep neural networks, $\map^{\ast}_i$, for $1 \leq i \leq I$, are independently trained, corresponding to each of the observables $\map_i$. The algorithm \ref{alg:dnnopt} and its analysis in section \ref{sec:an1}, can be readily extended to this case.  

Results of the DNNopt algorithm with randomly chosen $N_k$ training points, with $N_k = N_0 + k\Delta N$, with $N_0 = 64$ and $\Delta N = 16$, are presented in figure \ref{fig:heat1}. We observe from this figure that the mean $\bar{\gl}$ of the cost function \eqref{eq:std} barely decays with increasing number of training samples. Moreover, the standard deviation \eqref{eq:std} is also unacceptably high, even for a very high number of training samples. This is consistent with the bounds \eqref{eq:dn1}, \eqref{eq:dnk}, even though the hypotheses of section \ref{sec:an1} do not necessarily hold in this case. The bounds \eqref{eq:dn1}, \eqref{eq:dnk} suggest a very slow decay of the errors for this 10-dimensional ($d=10$) optimization problem.

Next, the ISMO algorithm \ref{alg:ismo} is run on this problem with the same setup as for the DNNopt algorithm, i.e., $I=5$ independent neural networks $\map^{\ast}_{k,i}$ are trained, each corresponding to the observable $\map_i$ \eqref{eq:htobs}, at every iteration $k$ of the ISMO algorithm. The results, with $N_0=64$ and $\Delta N = 16$ as hyperparameters of the ISMO algorithm are shown in figure \ref{fig:heat1}. We observe from this figure that the mean $\bar{\gl}$ of the cost function \eqref{eq:std} decays with increasing number of iterations, till a local minimum corresponding to a mean objective function value of approximately $0.05$ is reached. The standard deviation also decays as the number of iterations increases. In particular, the ISMO algorithm provides a significantly lower mean value of the cost function compared to the DNNopt algorithm. Moreover, the standard deviation is several times lower for the ISMO algorithm than for the DNNopt algorithm. Thus, we expect a robust parameter identification with the ISMO algorithm. This is indeed verified from figure \ref{fig:heat2}, where we plot the identified initial condition \eqref{eq:htin} with the ISMO algorithm with a starting size of 64 and a batch size of 16, and the corresponding solution field at time $T$ (computed with the Crank-Nicholson scheme), superposed with the measurements. We see from this figure that the identified initial data leads to one of the multiple solutions of this ill-posed inverse problem that approximates the measured values quite accurately.

This numerical experiment is openly avaiable at \url{https://github.com/kjetil-lye/ismo_heat}.
\begin{figure}[htbp]
    \begin{subfigure}{.48\textwidth}
        \centering
        \includegraphics[width=1\linewidth]{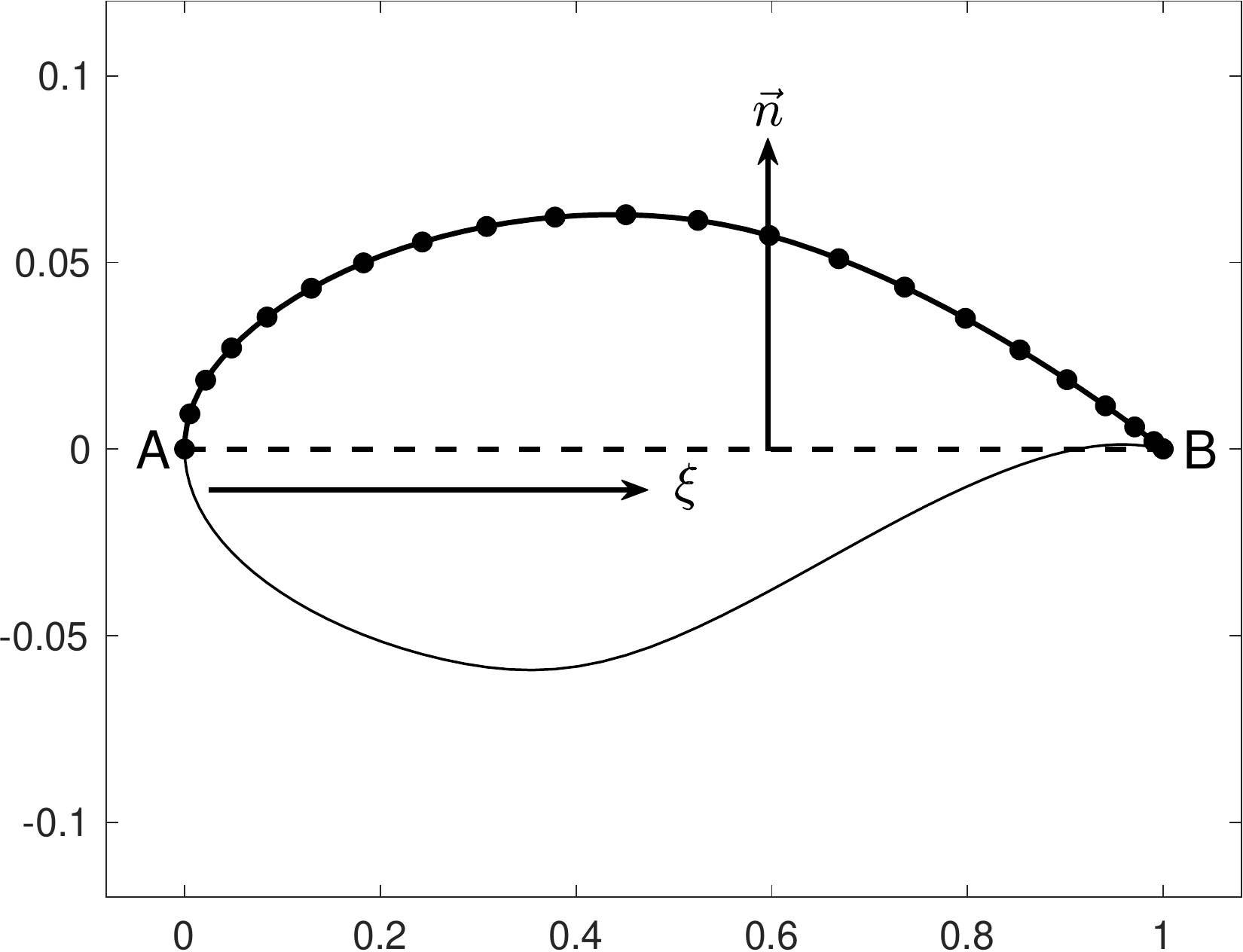}
        %\caption{Reference shape}
    \end{subfigure}
    \begin{subfigure}{.48\textwidth}
        \centering\
        \includegraphics[width=1\linewidth]{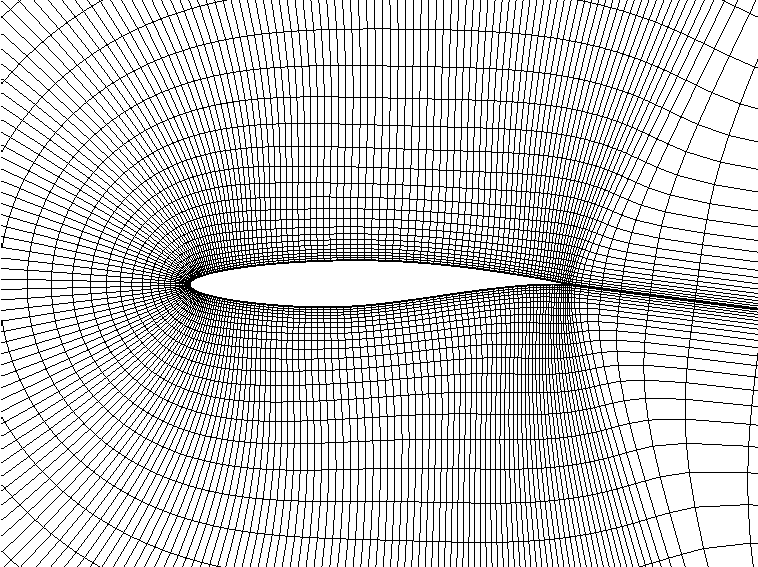}
       % \caption{Mach number (Sample)}
    \end{subfigure}
    \caption{The reference shape of the RAE2822 airfoil. Left: The parametrization used to deform the shape of the airfoil according to \eqref{eq:foil_deform}. Right: The initial C-grid mesh used around the airfoil.}
    \label{fig:ref_foil}
    \end{figure}
\subsection{Shape optimization of airfoils}
For the final numerical example in this paper, we choose optimization of the shapes of airfoils (wing cross-sections) in order to improve certain aerodynamic properties. The flow past the airfoil is modeled by the two-dimensional Euler equations, given by,
\begin{equation}
\label{eq:euler2D}
\U_t + \F^x(\U)_x + \F^y(\U)_y = 0,  \quad
\U = \begin{pmatrix}
\rho \\
\rho u \\
\rho v\\
E
\end{pmatrix}, \quad \F^x(\U) = \begin{pmatrix}
\rho u \\
\rho v^2 + p \\
\rho uv \\
(E + p) u
\end{pmatrix}, \quad \F^y(\U) = \begin{pmatrix}
\rho v \\
\rho uv \\
\rho v^2 + p \\
(E + p) v
\end{pmatrix}
\end{equation}
where $\rho, \vel = (u,v)$ and $p$ denote the fluid density, velocity and pressure, respectively. The quantity $E$ represents the total energy per unit volume
\[
E = \frac{1}{2} \rho |\vel|^2 + \frac{p}{\gamma -1},
\]
where $\gamma=c_p/c_v$ is the ratio of specific heats, chosen as $\gamma=1.4$ for our simulations. Additional important variables associated with the flow include the speed of sound $a = \sqrt{\gamma p/\rho}$ and the Mach number $M=|\vel|/a$.
\begin{figure}[htbp]
    \begin{subfigure}{.48\textwidth}
        \centering
        \includegraphics[width=1\linewidth]{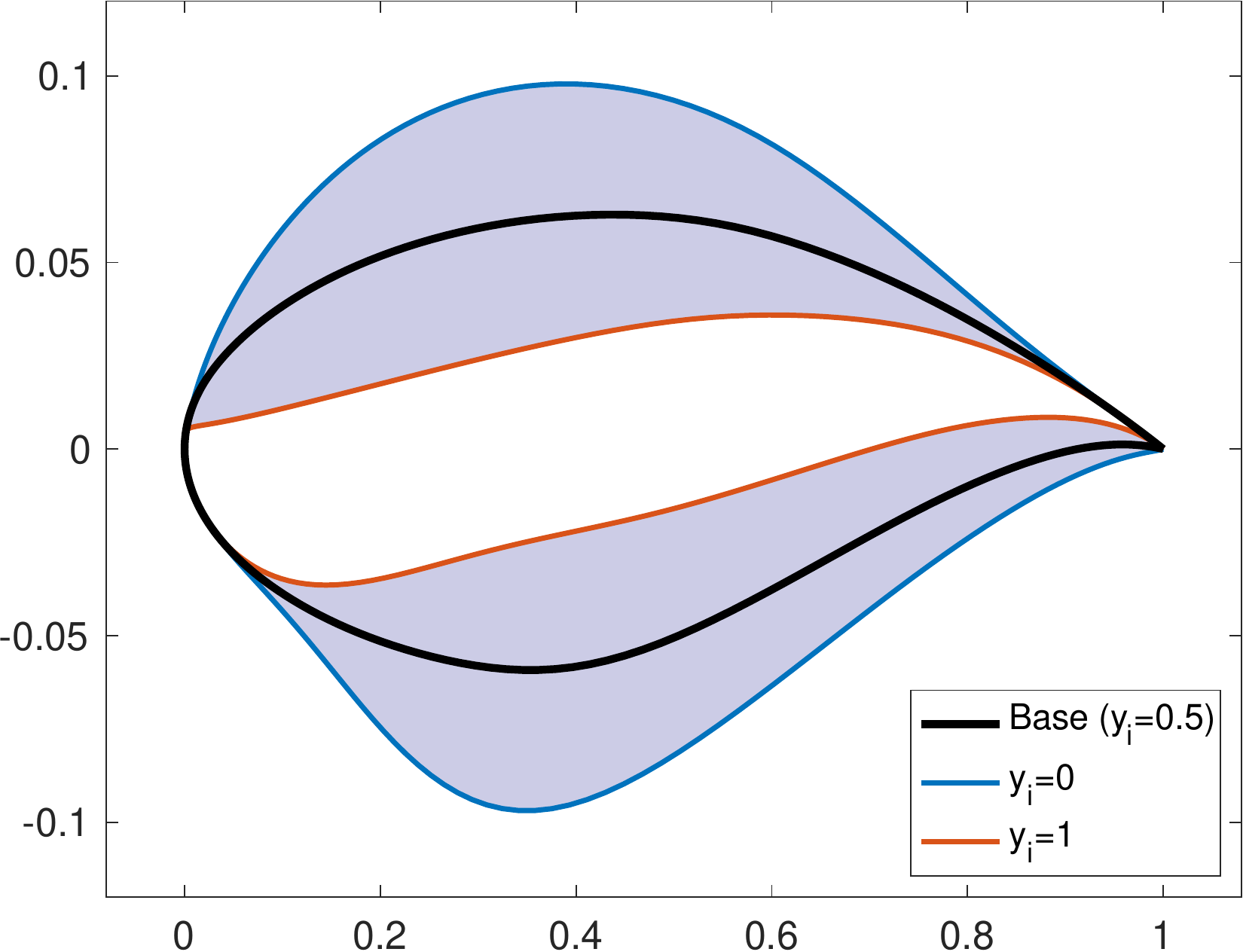}
        %\caption{Reference shape}
    \end{subfigure}
    \begin{subfigure}{.48\textwidth}
        \centering\
        \includegraphics[width=1\linewidth]{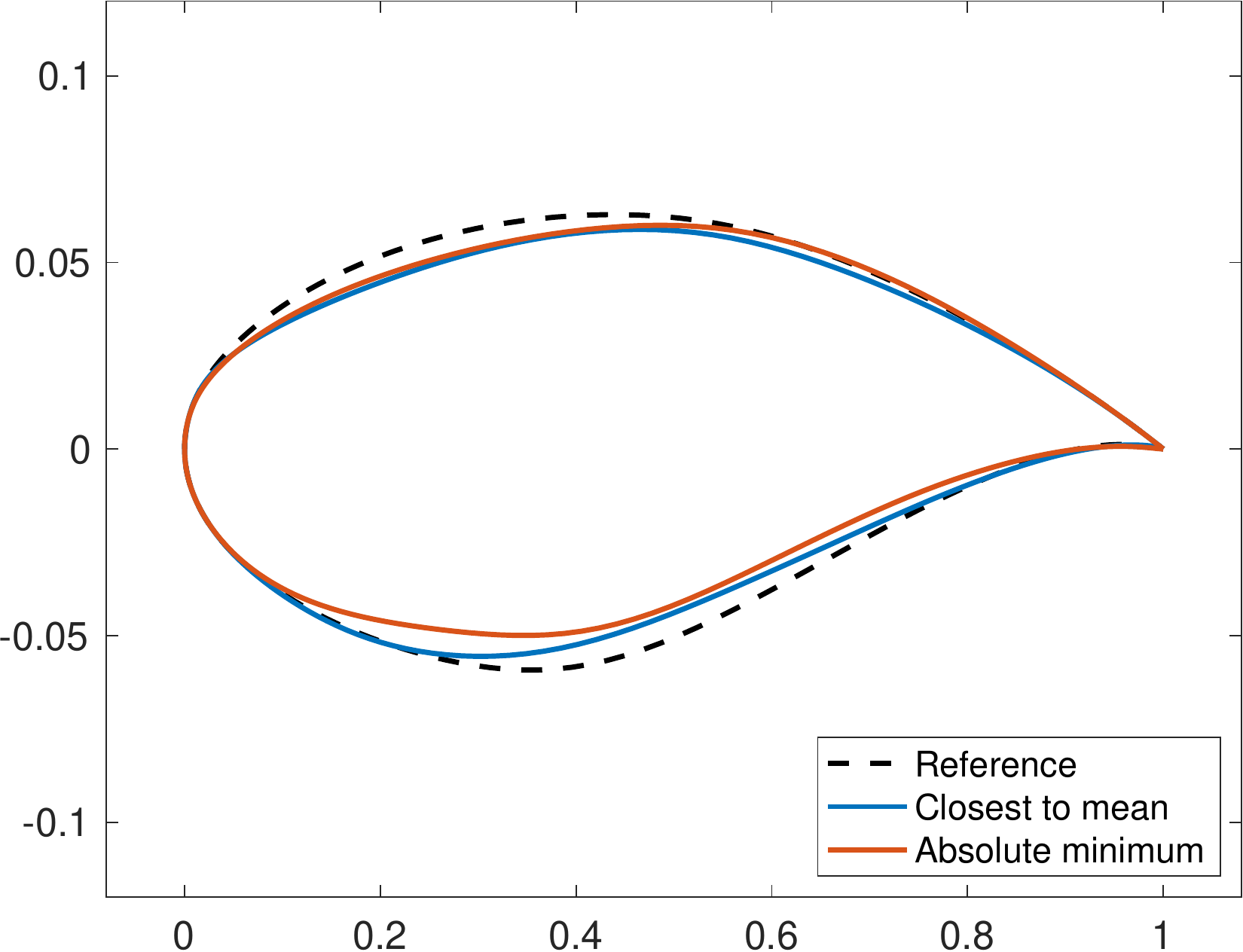}
       % \caption{Mach number (Sample)}
    \end{subfigure}
    \caption{Shapes for the airfoil shape optimization problem. Left: The design envelope, where the reference, widest and narrowest airfoils are obtained by setting all design parameters $y_i$ as 0.5, 0 and 1, respectively. Right: The optimized shapes obtained using the ISMO algorithm.}
    \label{fig:afl1}
    \end{figure}    

We will follow standard practice in aerodynamic shape optimization~\cite{Mohammadi2009} and consider a reference airfoil with upper and lower surface of the airfoil are located at $(x,y^{U,ref}(x/c))$ and $(x,y^{L,ref}(x/c))$, with $x \in [0,c]$, $c $ being the airfoil chord, and $y^{U,ref},y^{L,ref}$ corresponding to the well-known RAE2822 airfoil \cite{UMRIDA}, see figure \ref{fig:ref_foil} for a graphical representation of the reference airfoil.

The shape of this reference airfoil is perturbed in order to optimize aerodynamic properties. The upper and lower surface of the airfoil are perturbed as 
\begin{equation}
\label{eq:foil_deform}
    y = y^{ref}(\xi) + h(\xi), \quad h(\xi) = \sum_{k=1}^{d/2} a_k B_k(\xi), \qquad \xi = x/c,
\end{equation}
with $c$ being the airfoil chord length i.e., $c = |AB|$ for the airfoil depicted in figure \ref{fig:ref_foil}.

%For the upper surface of the airfoil, consider the line joining the leading edge A and trailing edge B parametrized by the local coordinate $\xi \in [0,1]$. Let $\vec{\bm{n}} = (n_x, n_y)$ be the unit vector perpendicular to this line, as shown in Figure \ref{fig:ref_foil}. Then a linear combination of smooth functions $B_k$ defined on line AB are used to deform the shape of the airfoil. More precisely, the coordinates $(x^{ref},y^{ref})$ on the upper surface are deformed along $\vec{\bm{n}}$
%\begin{equation}
  %  \label{eq:foil_deform}
   % \begin{pmatrix} x \\ y \end{pmatrix} = \begin{pmatrix} x^{ref} \\ y^{ref} \end{pmatrix} + \begin{pmatrix} n_x \\ n_y \end{pmatrix} h(\xi), \quad h(\xi) = \sum_{k=1}^{d/2} a_k B_k(\xi)
%\end{equation}
%where $a_k$ are the set of $d/2$ design parameters. We use a similar deformation strategy for the lower surface of the airfoil by parametrizing the line BA and flipping the unit vector $\vec{\bm{n}}$. Thus, we need $d$ design parameters in total.

%\textcolor{blue}{PC: Using normal vector n is useful if AB is inclined at an angle to x-axis. Here line AB is along x-axis, $n=(0,1)$, we can simplify the description like this
%\begin{equation}
 %   y = y^{U,ref}(\xi) + h(\xi), \quad h(\xi) = \sum_{k=1}^{d/2} a_k B_k(\xi), \qquad \xi = x/c
%\end{equation}
%where $c=|AB|$ is the airfoil chord length.
%}

Although several different parametrizations of shape or its  perturbations~\cite{Samareh2001,HH1} are proposed in the large body of aerodynamic shape optimization literature, we will focus on the simple and commonly used \emph{Hicks-Henne} bump functions \cite{HH1} to parametrize the perturbed airfoil shape
\begin{equation}
    \label{eq:hh_func}
    B_k(\xi) = \sin^3(\pi \xi^{q_k}), \quad q_k = \frac{\ln{2}}{ \ln(d+4) - \ln{k}}.
\end{equation}
Furthermore, the design parameters for each surface is chosen as
\begin{equation}\label{eq:deform_coef}
    \begin{aligned}
    a^{U}_k &= 2 (y^{U}_k - 0.5) \left( \frac{d}{2} - k + 1\right)\times 10^{-3}, \quad y^{U}_k \in [0,1],\\
    a^{L}_k &= 2 (y^L_k - 0.5) \left(k + 1\right)\times 10^{-3}, \quad y^L_k \in [0,1],
    \end{aligned}
\end{equation}
for $1 \leq k \leq d/2$. The transform functions in \eqref{eq:deform_coef} to obtain the design parameters have been chosen to ensure i) the perturbation is smaller near narrower tip of the airfoil (point B), and ii) the design envelope is large enough (see Figure \ref{fig:afl1}) while avoiding non-physical intersections of the top and bottom surfaces after deformation. Note that the airfoil shape is completely parametrized by the vector $y \in [0,1]^d$. For our experiments, we choose 10 parameters for each surface, i.e., we set $d=20$.

The flow around the airfoil is characterized by the free-stream boundary conditions corresponding to a Mach number of $M^\infty = 0.729$ and angle of attack of $\alpha = 2.31^{\circ}$. The observables of interest are the lift and drag coefficients given by, 
\begin{align}
\label{eq:lift}
    C_L(y) &= \frac{1}{K^\infty(y)} \int_{S} p(y)n(y) \cdot \hat{y} ds, \\
    \label{eq:drag}
    C_D(y) &= \frac{1}{K^\infty(y)} \int_{S} p(y)n(y) \cdot \hat{x} ds,
\end{align}
where $K^\infty(y) = \rho^\infty(y)\|\mathbf{u}^\infty(y)\|^2/2$ is the free-stream kinetic energy with $\hat{y} = [-\sin(\alpha),\cos(\alpha)]$ and $\hat{x} = [\cos(\alpha),\sin(\alpha)]$. 

\begin{figure}[htbp]
\centering
\includegraphics[width=8cm]{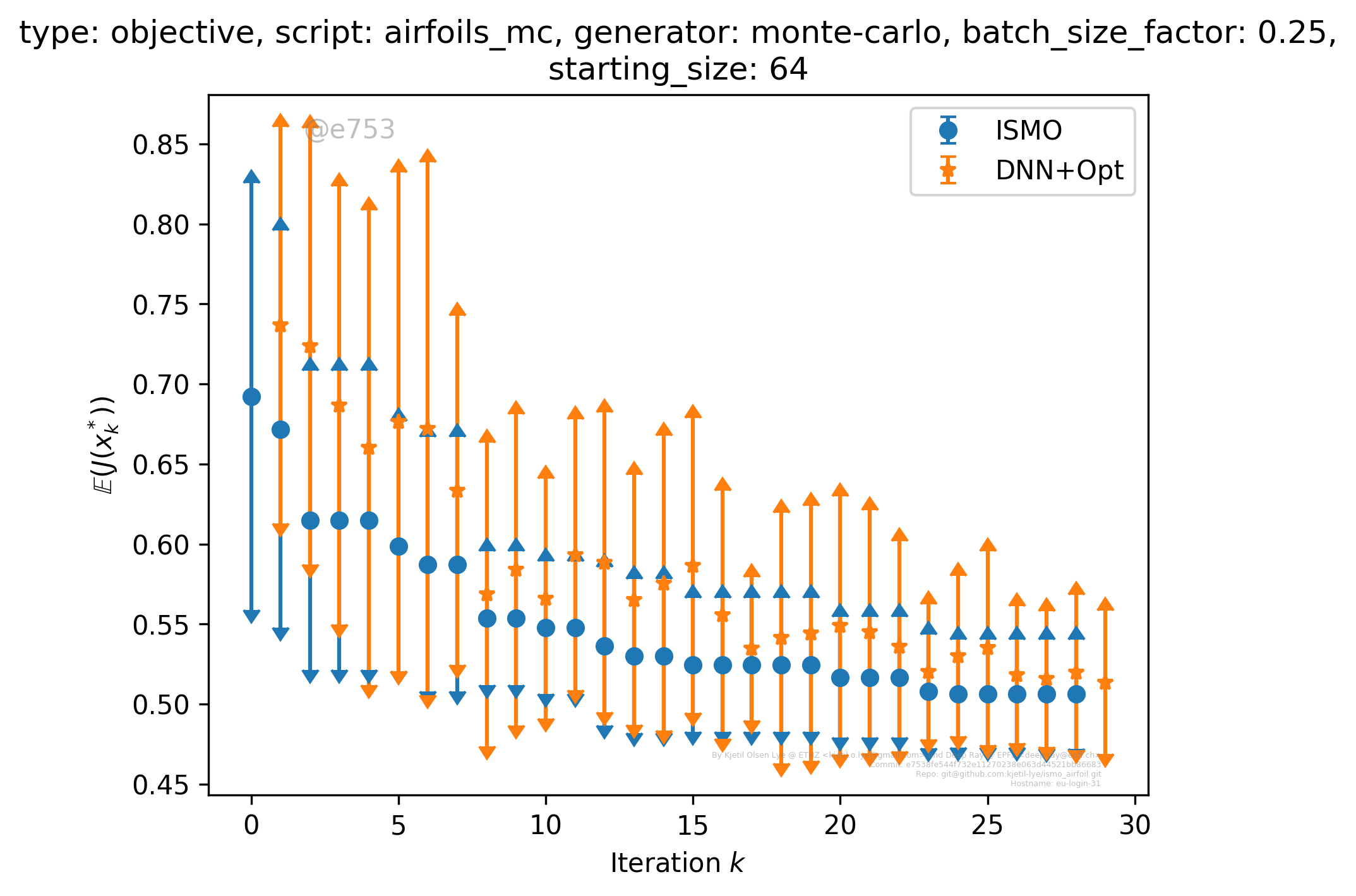}
\caption{The objective function \eqref{eq:aflct} vs. number of iterations  of the DNNopt and ISMO algorithms for the airfoil shape optimization problem. The mean of the objective function and mean $\pm$ standard deviation \eqref{eq:std} are shown.}
\label{fig:afl2}
\end{figure}

As is standard in aerodynamic shape optimization~\cite{Reuther1995,Mohammadi2009,Lyu2015}, we will modulate airfoil shape in order to \emph{minimize drag while keeping lift constant}. This leads to the following optimization problem,
\begin{equation}
    \label{eq:aflct}
    \gl(y) =  \frac{C_D(y)}{C_D^{ref}} + P\max\left(0, 0.999-\frac{C_L}{C_L^{ref}}\right)
\end{equation}
with $C_{L,D}$ being the lift \eqref{eq:lift} and drag \eqref{eq:drag} coefficients, respectively, and $C_L^{ref}$ is the lift corresponding to the reference RAE2822 airfoil and $P = 10000$ is a parameter to penalize any deviations from the lift of the reference airfoil. The reference lift is calculated to be $C_L^{ref} = 0.8763$, and the reference drag is $C_D^{ref}=0.011562$.

We will compute approximate minimizers of the cost function \eqref{eq:aflct} with the DNNopt algorithm \ref{alg:dnnopt} and the ISMO algorithm \ref{alg:ismo}. The training data is generated using the NUWTUN solver\footnote{http://bitbucket.org/cpraveen/nuwtun} to approximate solutions of the Euler equations \eqref{eq:euler2D} around the deformed airfoil geometry \eqref{eq:foil_deform}, with afore-mentioned boundary conditions. The NUWTUN code is based on the ISAAC code\footnote{http://isaac-cfd.sourceforge.net}~\cite{Morrison1992} and employs a finite volume solver using the Roe numerical flux and second-order MUSCL reconstruction with the Hemker-Koren limiter. Once the airfoil shape is perturbed, the base mesh (see Figure \ref{fig:ref_foil}) is deformed to fit the new geometry using a thin plate splines based radial basis function interpolation technique~\cite{Kumar2007,Duvigneau2011}, which is known to generate high quality grids. The problem is solved to steady state using an implicit Euler time integration. 

First, we run the DNNopt algorithm \ref{alg:dnnopt}, with two independent neural networks $C_L^{\ast}$ and $C_D^{\ast}$ for the lift coefficient $C_L$ and drag coefficient $C_D$, respectively. The results of the DNNopt algorithm, with randomly chosen training points in $Y = [0,1]^{20}$, and with $N_0+k\Delta N$ training samples, for $N_0=64$ and $\Delta N = 16$, are presented in figure \ref{fig:afl2}. From this figure, we observe that the mean $\bar{\gl}$ of the cost function \eqref{eq:aflct} decays, but in an oscillatory manner, with increasing number of training samples. Moreover, the standard deviation \eqref{eq:std} also decays but is still rather high, even for a large number of iterations. This is consistent with the findings in the other two numerical experiments reported here. These results suggest that the DNNopt algorithm may not yield robust or accurate optimal shapes. 

Next, we run the ISMO algorithm \ref{alg:ismo} with the same setup as DNNopt, i.e.,  with two independent neural networks $C_{L,k}^{\ast}, C_{D,k}^{\ast}$, at every iteration $k$, for the lift and the drag, respectively. We set hyperparameters $N_0 = 64$ and batch size $\Delta N = 16$ and present results with the ISMO algorithm in figure \ref{fig:afl2}. We see from this figure that the mean and the standard deviation of the cost function steadily decay with increasing number of iterations for the ISMO algorithm. Moreover, the mean as well as the standard deviation of cost function, approximated by the ISMO algorithm, are significantly less than those by the DNNopt algorithm. Thus, one can expect that the ISMO algorithm provides parameters for robust shape optimization of the airfoil. This is further verified in figure \ref{fig:afl3}, where we separately plot the mean and mean $\pm$ standard deviation for the drag and lift coefficients, corresponding to the optimizers at the each iteration, for the DNNopt and ISMO algorithms. From this figure, we observe that the mean and standard deviation for the computed drag, with the ISMO algorithm, decay with increasing number of iterations. Moreover, both the mean and standard deviation of the drag are significantly lower than those computed with the DNNopt algorithm. On the other hand, both algorithms result in Hicks-Henne parameters that keep lift nearly constant (both in mean as well as in mean $\pm$ standard deviation), around a value of $C_L = 0.90$, which is slightly higher than the reference lift. This behavior is completely consistent with the structure of the cost function \eqref{eq:aflct}, which places the onus on minimizing the drag, while keeping lift nearly constant. 
\begin{figure}[htbp]
    \begin{subfigure}{.48\textwidth}
        \centering
        \includegraphics[width=1\linewidth]{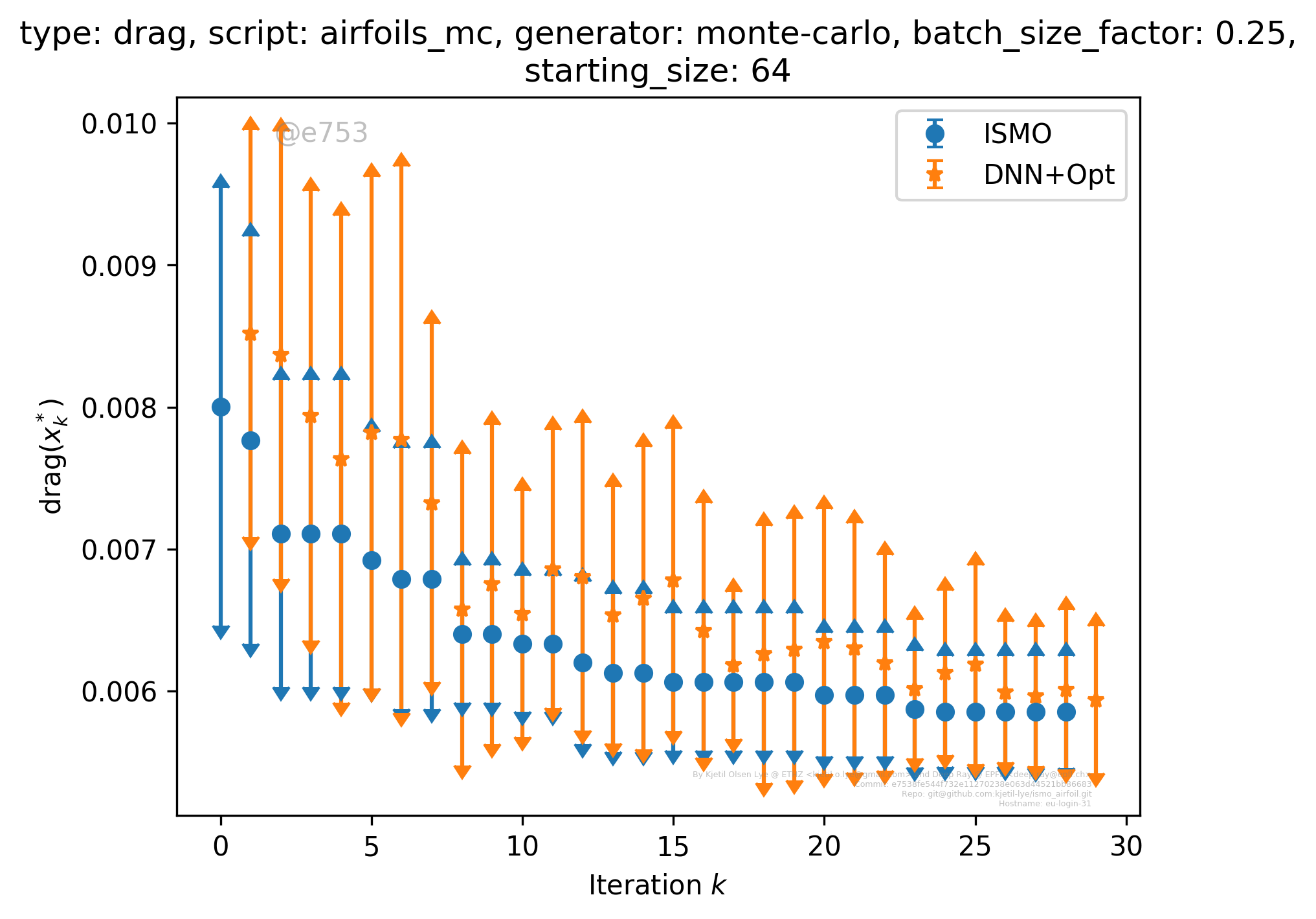}
        %\caption{Reference shape}
    \end{subfigure}
    \begin{subfigure}{.48\textwidth}
        \centering\
        \includegraphics[width=1\linewidth]{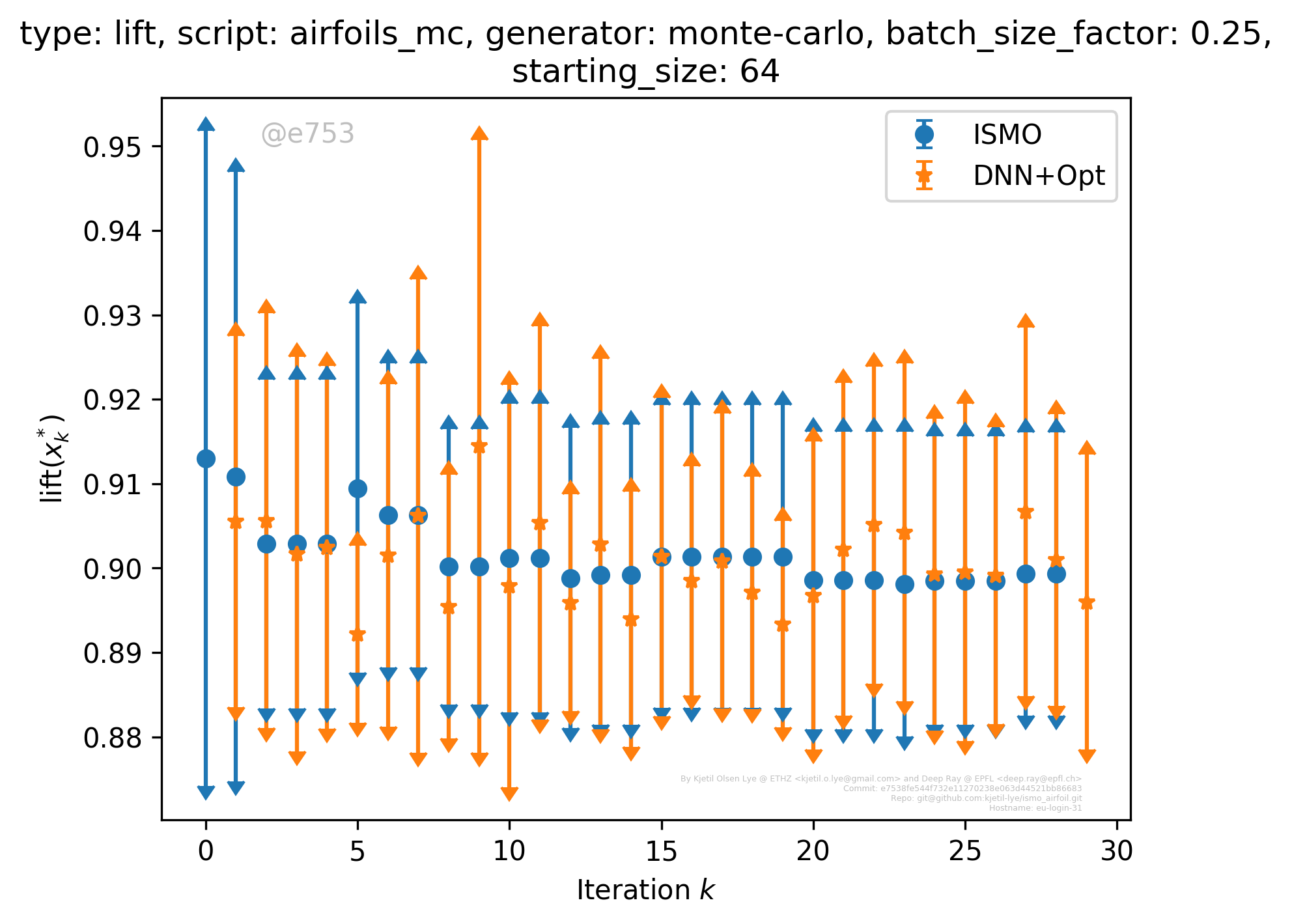}
       % \caption{Mach number (Sample)}
    \end{subfigure}
    \caption{The Drag \eqref{eq:drag} (Left) and Lift \eqref{eq:lift} (Right) vs. number of iterations  of the DNNopt and ISMO algorithms for the airfoil shape optimization problem. The mean of the objective function and mean $\pm$ standard deviation \eqref{eq:std} are shown.}
    \label{fig:afl3}
    \end{figure}

Quantitatively, the ISMO algorithm converges to a set of parameters (see the resulting shape in figure \ref{fig:afl1}) that correspond to the computed minimum drag coefficient of $C_D = 0.005202$, when compared to the drag $C_D^{ref} = 0.01156$, of the reference RAE2822 airfoil. The corresponding lift coefficient with this optimizer is $C_L = 0.8960$, when compared to the reference lift $C_L^{ref} = 0.8763$. Similarly, the computed Hicks-Henne parameters, with the ISMO algorithm, that lead to a cost function \eqref{eq:aflct} that is closest to the mean cost function $\bar{\gl}$ (see shape of the resulting airfoil in figure \ref{fig:afl1}), resulted in a drag coefficient of $C_D = 0.005828$ and lift coefficient of $C_L = 0.8879$. Thus, while the optimized lift was kept very close to the reference lift, the drag was reduced by approximately $55\%$ (for the absolute optimizer) and $50\%$ (on an average) by the shape parameters, computed with the ISMO algorithm. 

In order to investigate the reasons behind this significant drag reduction, we plot the Mach number around the airfoil to visualize the flow, for the reference airfoil and two of the optimized airfoils (one corresponding to the absolute smallest drag and the other to the mean optimized drag) in figure \ref{fig:afl3}. From this figure we see that the key difference between the reference and optimized airfoil is significant reduction (even elimination) in the strength of the shock over the upper surface of the airfoil. This diminishing of shock strength reduces the shock drag and the overall cost function \eqref{eq:aflct}.

Finally, we compare the performance of the ISMO algorithm for this example with a standard optimization algorithm. To this end, we choose a black box truncated Newton (TNC) algorithm, with its default implementation in \emph{SciPy}~\cite{2020SciPy-NMeth}. In figure \ref{fig:afl5}, we plot the mean and the mean $\pm$ standard deviation of the cost function with the TNC algorithm, with 1024 starting values chosen randomly and compare it with the ISMO algorithm \ref{alg:ismo}. From figure \ref{fig:afl5} (left), we observe that the ISMO algorithm has a significantly lower mean of the objective function, than the black box TNC algorithm, even for a very large number ($2^{11}-2^{12}$) of evaluations (calls) to the underlying PDE solver. Clearly, the ISMO algorithm readily outperforms the TNC algorithm by more than an order of magnitude, with respect to the mean of the objective function \eqref{eq:aflct}. Moreover, we also observe from figure \ref{fig:afl5} (right) that the ISMO algorithm has significantly (several orders of magnitude) lower standard deviation than the TNC algorithm for the same number of calls to the PDE solver. Thus, this example clearly illustrates that ISMO algorithm is both more efficient as well as more robust than standard optimization algorithms such as the truncated Newton methods.    

This numerical experiment is openly avaiable at \url{https://github.com/kjetil-lye/ismo_airfoil}.

\section{Discussion}
\label{sec:6}
Robust and accurate numerical approximation of the solutions of PDE constrained optimization problems represents a formidable computational challenge as existing methods might require a large number of evaluations of the solution (and its gradients) for the underlying PDE. Since each call to the PDE solver is computationally costly, multiple calls can be prohibitively expensive.

\begin{figure}[htbp]
    \begin{subfigure}{.32\textwidth}
        \centering
        \includegraphics[width=1\linewidth]{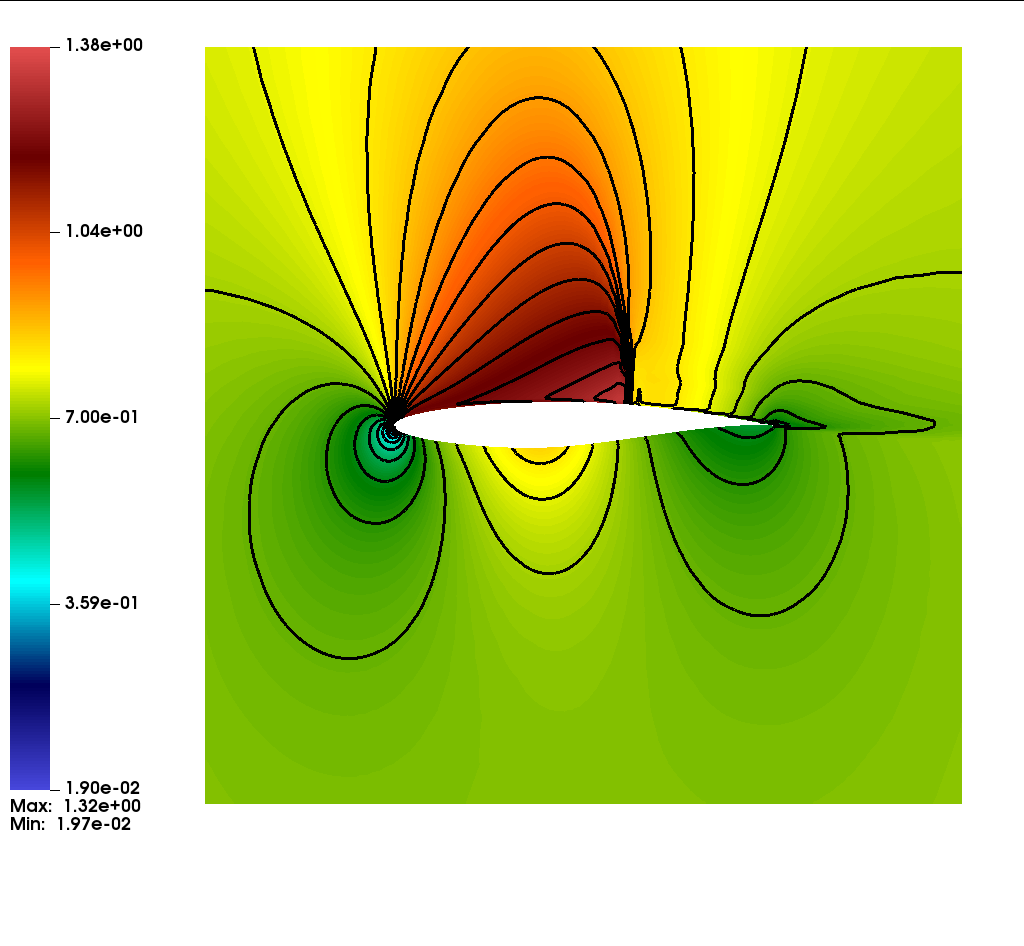}
        %\caption{Reference shape}
    \end{subfigure}
    \begin{subfigure}{.32\textwidth}
        \centering\
        \includegraphics[width=1\linewidth]{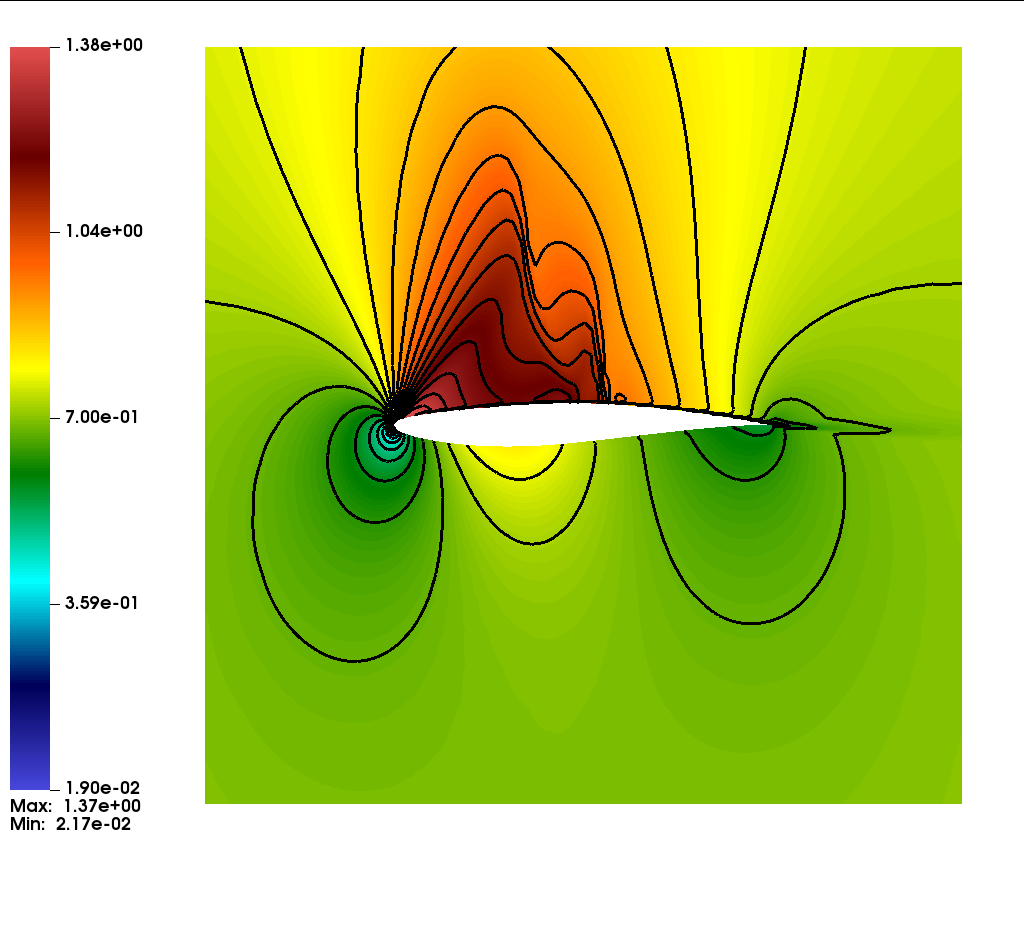}
       % \caption{Mach number (Sample)}
    \end{subfigure}
    \begin{subfigure}{.32\textwidth}
        \centering\
        \includegraphics[width=1\linewidth]{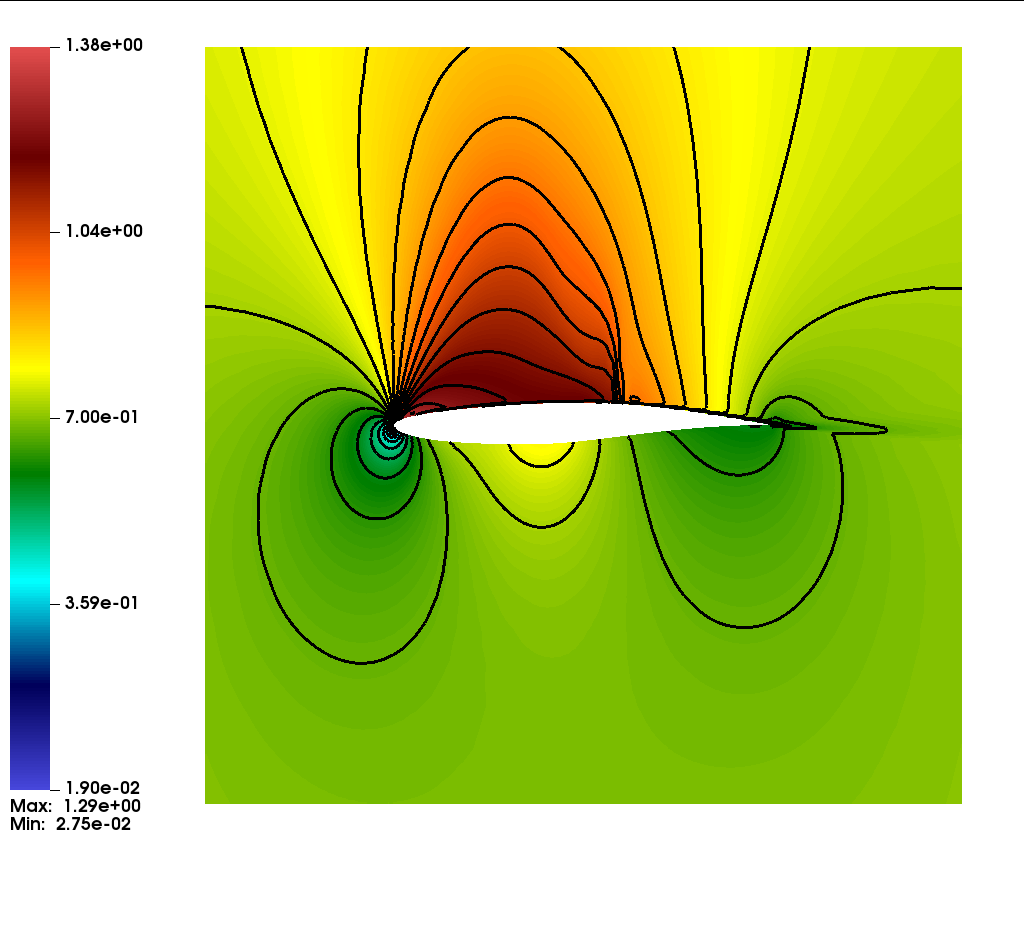}
       % \caption{Mach number (Sample)}
    \end{subfigure}
    \caption{Mach number contours for the flow around airfoils. Left: Flow around reference RAE2822 airfoil. Center: Flow around optimized airfoil, corresponding to Hicks-Henne parameters with a cost function \eqref{eq:aflct} that is closest of the (converged) mean cost function for the ISMO algorithm. Right: Flow around optimized airfoil, corresponding to Hicks-Henne parameters with the smallest (converged) cost function for the ISMO algorithm}
    \label{fig:afl4}
    \end{figure}

Using surrogate models within the context of PDE constrained optimization is an attractive proposition. As long as the surrogate provides a robust and accurate approximation to the underlying PDE solution, while being much cheaper to evaluate computationally, surrogate models can be used inside standard optimization based algorithms, to significantly reduce the computational cost. However, finding surrogates with desirable properties can be challenging for high dimensional problems. 

Deep neural networks (DNNs) have emerged as efficient surrogates for PDEs, particularly for approximating \emph{observables} of PDEs, see \cite{LMR1,LMM1,MR1} and references therein. Thus, it is natural to investigate whether DNNs can serve as efficient surrogates in the context of PDE constrained optimization. To this end, we propose the DNNopt algorithm \ref{alg:dnnopt} that combines standard optimization algorithms, such as the quasi-Newton algorithm \ref{alg:qn}, with deep neural network surrogates. 

A careful theoretical analysis of the DNNopt algorithm, albeit in a restricted setting presented in section \ref{sec:an1}, reveals a fundamental issue with it, namely the accuracy and robustness (measured in terms of the range and standard deviation \eqref{eq:std} with respect to starting values) suffers from a significant \emph{curse of dimensionality}. In particular, estimates \eqref{eq:dnbd}, \eqref{eq:dnbd1}, \eqref{eq:dnk1} show that the error (and its standard deviation) with the DNNopt algorithm only decays very slowly with increasing number of training samples. This slow decay and high variance are verified in several numerical examples presented here.

We find that fixing training sets \emph{a priori} for the deep neural networks might lead to poor approximation of the underlying minima for the optimization problem, which correspond to a subset (manifold) of the parameter space. To alleviate this shortcoming of the DNNopt algorithm, we propose a novel algorithm termed as iterative surrogate model optimization (ISMO) algorithm \ref{alg:ismo}. The key idea behind ISMO is to iteratively augment the training set for a sequence of deep neural networks such that the underlying minima of the optimization problem are better approximated. At any stage of the ISMO iteration, the current state of the DNN is used as an input to the standard optimization algorithm to find (approximate) local minima for the underlying cost function. These local minima are added to the training set and DNNs for the next step are trained on the augmented set. This feedback algorithm is iterated till convergence. Thus, one can view ISMO as an \emph{active learning algorithm} \cite{AL}, where the learner (the deep neural network) queries the teacher or oracle (standard optimization algorithm \ref{alg:qn}) to provide a better training set at each iteration. 

\begin{figure}[htbp]
\centering
\begin{subfigure}{0.49\textwidth}
\includegraphics[width=\textwidth]{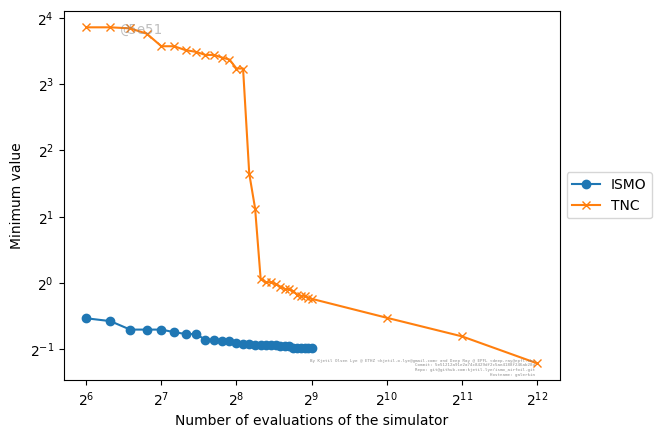}
\caption{Mean}
\end{subfigure}
\begin{subfigure}{0.49\textwidth}
\includegraphics[width=\textwidth]{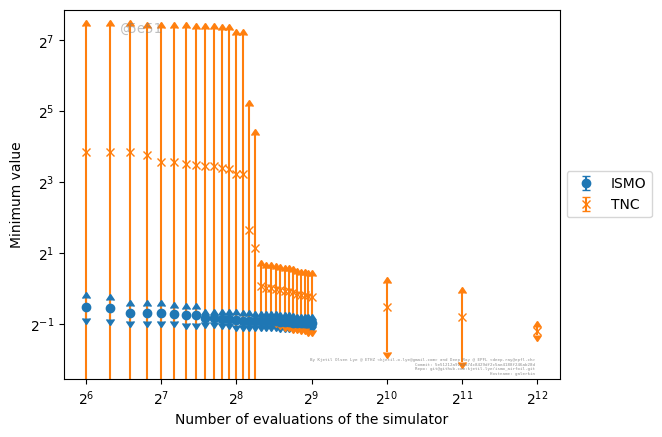}
\caption{Mean $\pm$ Standard Deviation}
\end{subfigure}
\caption{Comparison of ISMO with the black box TNC algorithm of \cite{2020SciPy-NMeth} for the airfoil shape optimization example. Left: Mean of the cost function \eqref{eq:aflct} vs. number of calls to the PDE solver. Right: Mean $\pm$ Standard deviation \eqref{eq:std} for the cost function \eqref{eq:aflct}}
\label{fig:afl5}
\end{figure}

In section \ref{sec:an2}, we analyze the ISMO algorithm in a restricted setting and prove that the resulting (approximate) optimizers converge to the underlying minimum of the optimization problem, \emph{exponentially} in terms of the number of iterations of the algorithm, see estimates \eqref{eq:isbd}. Moreover, the underlying variance is also reduced exponentially \eqref{eq:isbd1}. This exponential convergence \eqref{eq:isk1} should be contrasted with the algebraic convergence for the DNNopt algorithm \eqref{eq:dnk1} and can ameliorate the curse of dimensionality.

Although the theoretical results were proved in a restricted setting, we validated the ISMO algorithm and compared it with the DNNopt algorithm for three representative numerical examples, namely for an optimal control problem for a nonlinear ODE, a data assimilation (parameter identification) inverse problem for the heat equation and shape optimization of airfoils, subject to the Euler equations of compressible fluid dynamics. For all these examples, the ISMO algorithm was shown to significantly outperform the DNNopt algorithm, both in terms of the decay of the (mean) objective function and its greatly reduced sensitivity to starting values. Moreover, the ISMO algorithm outperformed (by more than an order of magnitude) a standard black-box optimization algorithm for aerodynamic shape optimization.

Thus, on the basis of the proposed theory and numerical examples, we can assert that the active learning ISMO algorithm provides a very promising framework for significantly reducing the computational cost of PDE constrained optimization problems, while being robust and accurate. Moreover, it is very flexible, with respect to the choice of the underlying optimization algorithm as well as architecture of neural networks, and is straightforward to implement in different settings.

This article is the first to present the ISMO algorithm and readily lends itself to the following extensions.
\begin{itemize}
    \item We have presented the ISMO algorithm in a very general setting of an abstract PDE constrained optimization problem. Although we have selected three representative examples in this article, the algorithm itself can be applied in a straightforward manner to a variety of PDEs, both linear and non-linear. 
    \item We have focused on a particular type of underlying optimization algorithm, i.e.,  of the quasi-Newton type. However, the ISMO algorithm \ref{alg:ismo} does not rely on any specific details of the optimization algorithm. Hence, any optimization algorithm can be used as the teacher or oracle within ISMO. In particular, one can also use gradient-free optimization algorithms such as particle swarm optimization and genetic algorithms within ISMO. Thus, even black-box optimization algorithms can employed within ISMO to significantly accelerate finding solutions of PDE constrained optimization problems. 
    \item Similarly, the ISMO algorithm does not rely on a specific structure of the surrogate model. Although we concentrated on neural networks in this article, one can readily use other surrogates such as Gaussian process regression, within the ISMO algorithm.
    \item In this article, we have considered the PDE constrained optimization problem, where the objective function in \eqref{eq:cost1}, is defined in terms of an \emph{observable} \eqref{eq:ptoob} of the PDE solution. In some problems, particularly for control problems, it might be necessary to consider the whole solution field and approximate it with a neural network. Such DNNs are also available but might be more expensive to train and evaluate. Thus, one needs to carefully investigate the resulting speedup with the ISMO algorithm over standard optimization algorithms in this context.
    \item A crucial issue in PDE constrained optimization is that of optimization under uncertainty, see \cite{SS1} and references therein. In this context, the ISMO algorithm demonstrates considerable potential to significantly outperform state of the art algorithms and we consider this extension in a forthcoming paper.

\end{itemize}
\section*{Acknowledgements}
The research of SM is partially supported by ERC CoG 770880 SPARCCLE. The research of PC is supported by Department of Atomic Energy, Government of India, under  project no. 12-R\&D-TFR-5.01-0520.

\bibliographystyle{abbrv}
\bibliography{ISMOpaper_ref}

\begin{thebibliography}{10}

\bibitem{tensorflow2015-whitepaper}
M.~Abadi, A.~Agarwal, P.~Barham, E.~Brevdo, Z.~Chen, C.~Citro, G.~S. Corrado,
  A.~Davis, J.~Dean, M.~Devin, S.~Ghemawat, I.~Goodfellow, A.~Harp, G.~Irving,
  M.~Isard, Y.~Jia, R.~Jozefowicz, L.~Kaiser, M.~Kudlur, J.~Levenberg,
  D.~Man\'{e}, R.~Monga, S.~Moore, D.~Murray, C.~Olah, M.~Schuster, J.~Shlens,
  B.~Steiner, I.~Sutskever, K.~Talwar, P.~Tucker, V.~Vanhoucke, V.~Vasudevan,
  F.~Vi\'{e}gas, O.~Vinyals, P.~Warden, M.~Wattenberg, M.~Wicke, Y.~Yu, and
  X.~Zheng.
\newblock {TensorFlow}: Large-scale machine learning on heterogeneous systems,
  2015.
\newblock Software available from tensorflow.org.

\bibitem{Jent1}
C.~Beck, S.~Becker, P.~Grohs, N.~Jaafari, and A.~Jentzen.
\newblock Solving stochastic differential equations and kolmogorov equations by
  means of deep learning.
\newblock Preprint, available as arXiv:1806.00421v1.

\bibitem{BV1}
R.~Becker, D.~Meidner, and B.~Vexler.
\newblock Efficient numerical solution of parabolic optimization problems by
  finite element methods.
\newblock {\em Optimization methods and sofware}, 2007.

\bibitem{BS1}
A.~Borzi and V.~Schultz.
\newblock {\em Computational optimization of systems governed by partial
  differential equations}.
\newblock SIAM, 2012.

\bibitem{CAF1}
R.~E. Caflisch.
\newblock Monte carlo and quasi-monte carlo methods.
\newblock {\em Acta Numerica}, 7:1--49, 1998.

\bibitem{CALD}
J.~Calder.
\newblock Consistency of lipschitz learning with infinite unlabeled data and
  finite labeled data,.
\newblock {\em SIAM Journal on Mathematics of Data Science}, 1:780--812, 2019.

\bibitem{chollet2015keras}
F.~Chollet et~al.
\newblock Keras.
\newblock \url{https://keras.io}, 2015.

\bibitem{PSO}
M.~Clerc.
\newblock {\em Particle swarm optimization}.
\newblock Wiley, 2005.

\bibitem{CS1}
F.~Cucker and S.~Smale.
\newblock On the mathematical foundations of learning.
\newblock {\em Bulletin of the American Mathematical Society}, 39(1):1--49,
  2002.

\bibitem{Duvigneau2011}
R.~Duvigneau and P.~Chandrashekar.
\newblock Kriging-based optimization applied to flow control.
\newblock {\em International Journal for Numerical Methods in Fluids},
  69(11):1701--1714, Aug. 2011.

\bibitem{HEJ1}
W.~E, J.~Han, and A.~Jentzen.
\newblock Deep learning-based numerical methods for high-dimensional parabolic
  partial differential equations and backward stochastic differential
  equations.
\newblock {\em Communications in Mathematics and Statistics}, 5(4):349--380,
  2017.

\bibitem{Dfold}
R.~Evans, J.~Jumper, J.~Kirkpatrick, L.~Sifre, T.~Green, C.~Qin, A.~Zidek,
  A.~Nelson, A.~Bridgland, H.~Penedones, et~al.
\newblock De novo structure prediction with deep-learning based scoring.
\newblock {\em Annual Review of Biochemistry}, 77(6):363--382, 2018.

\bibitem{lbfgs}
R.~Fletcher.
\newblock {\em Practical methods of optimization}.
\newblock John Wiley and sons, 1987.

\bibitem{Forrester2008}
A.~I.~J. Forrester, A.~Sóbester, and A.~J. Keane.
\newblock {\em Engineering {Design} via {Surrogate} {Modelling}: {A}
  {Practical} {Guide}}.
\newblock Wiley, 2008.

\bibitem{DLbook}
I.~Goodfellow, Y.~Bengio, and A.~Courville.
\newblock {\em Deep learning}.
\newblock MIT press, 2016.

\bibitem{E1}
J.~Han, A.~Jentzen, and W.~E.
\newblock Solving high-dimensional partial differential equations using deep
  learning.
\newblock {\em Proceedings of the National Academy of Sciences},
  115(34):8505--8510, 2018.

\bibitem{UMRIDA}
C.~Hirsch, D.~Wunsch, J.~Szumbarski, J.~Pons-Prats, et~al.
\newblock Uncertainty management for robust industrial design in aeronautics.
\newblock Notes on Numerical Fluid Mechanics and Multidisciplinary Design.
  Springer, 2019.

\bibitem{IBM-Spectrum-LSF}
IBM.
\newblock Ibm spectrum ldf, 2020 (accessed July 27, 2020).

\bibitem{adam}
D.~P. Kingma and J.~Ba.
\newblock Adam: {A} method for stochastic optimization.
\newblock In {\em 3rd International Conference on Learning Representations,
  {ICLR} 2015}, 2015.

\bibitem{Kumar2007}
K.~Kumar and M.~T. Nair.
\newblock A mesh deformation strategy for multiblock structured grids.
\newblock Project {Document} CF 0708, National Aerospace Laboratories, 2007.

\bibitem{Lag1}
I.~E. Lagaris, A.~Likas, and D.~I. Fotiadis.
\newblock Artificial neural networks for solving ordinary and partial
  differential equations.
\newblock {\em IEEE Transactions on Neural Networks}, 9(5):987--1000, 1998.

\bibitem{DLnat}
Y.~LeCun, Y.~Bengio, and G.~Hinton.
\newblock Deep learning.
\newblock {\em Nature}, 521(7553):436--444, 2015.

\bibitem{LMM1}
K.~O. Lye, S.~Mishra, and R.~Molinaro.
\newblock A multi-level procedure for enhancing accuracy of machine learning
  algorithms.
\newblock {\em European Journal of Applied Mathematics}, 2020.

\bibitem{LMR1}
K.~O. Lye, S.~Mishra, and D.~Ray.
\newblock Deep learning observables in computational fluid dynamics.
\newblock {\em Journal of Computational Physics}, page 109339, 2020.

\bibitem{Lyu2015}
Z.~Lyu, G.~K.~W. Kenway, and J.~R. R.~A. Martins.
\newblock Aerodynamic shape optimization investigations of the common research
  model wing benchmark.
\newblock {\em AIAA Journal}, 53(4):968--985, 2015.

\bibitem{HH1}
D.~A. Masters, N.~J. Taylor, T.~Rendall, C.~B. Allen, and D.~J. Poole.
\newblock Geometric comparison of aerofoil shape parameterization methods.
\newblock {\em AIAA Journal}, pages 1575--1589, 2017.

\bibitem{SM1}
S.~Mishra.
\newblock A machine learning framework for data driven acceleration of
  computations of differential equations.
\newblock {\em Mathematics in Engineering}, 1:118, 2019.

\bibitem{MM1}
S.~Mishra and R.~Molinaro.
\newblock Estimates on the generalization error of physics informed neural
  networks (pinns) for approximating pdes.
\newblock Preprint, available from arXiv:2006:16144v1, 2020.

\bibitem{MM2}
S.~Mishra and R.~Molinaro.
\newblock Estimates on the generalization error of physics informed neural
  networks (pinns) for approximating pdes ii: A class of inverse problems.
\newblock Preprint, available from ETH Z\"urich SAM Reports, 2020.

\bibitem{MR1}
S.~Mishra and T.~K. Rusch.
\newblock Enhancing accuracy of deep learning algorithms by training with
  low-discrepancy sequences.
\newblock Preprint, available as arXiv:2005.12564, 2020.

\bibitem{Mohammadi2009}
B.~Mohammadi and O.~Pironneau.
\newblock {\em Applied {Shape} {Optimization} for {Fluids}}.
\newblock Oxford University Press, 2009.

\bibitem{Morrison1992}
J.~H. Morrison.
\newblock A compressible {Navier}-{Stokes} solver with two-equation and
  {Reynolds} stress turbulence closure models.
\newblock {NASA} {Contractor} {Report} 4440, NASA, 1992.

\bibitem{owen}
A.~B. Owen.
\newblock Multidimensional variation for quasi-monte carlo.
\newblock In {\em Contemporary Multivariate Analysis And Design Of Experiments:
  In Celebration of Professor Kai-Tai Fang's 65th Birthday}, pages 49--74.
  World Scientific, 2005.

\bibitem{ROMbook}
A.~Quarteroni, A.~Manzoni, and F.~Negri.
\newblock {\em Reduced basis methods for partial differential equations: an
  introduction}, volume~92.
\newblock Springer, 2015.

\bibitem{KAR1}
M.~Raissi and G.~E. Karniadakis.
\newblock Hidden physics models: Machine learning of nonlinear partial
  differential equations.
\newblock {\em Journal of Computational Physics}, 357:125--141, 2018.

\bibitem{KAR2}
M.~Raissi, P.~Perdikaris, and G.~E. Karniadakis.
\newblock Physics-informed neural networks: A deep learning framework for
  solving forward and inverse problems involving nonlinear partial differential
  equations.
\newblock {\em Journal of Computational Physics}, 378:686--707, 2019.

\bibitem{KAR4}
M.~Raissi, A.~Yazdani, and G.~E. Karniadakis.
\newblock Hidden fluid mechanics: A navier-stokes informed deep learning
  framework for assimilating flow visualization data.
\newblock {\em arXiv preprint arXiv:1808.04327}, 2018.

\bibitem{GPRbook}
C.~E. Rasmussen.
\newblock Gaussian processes in machine learning.
\newblock In {\em Summer School on Machine Learning}, pages 63--71. Springer,
  2003.

\bibitem{DR1}
D.~Ray and J.~S. Hesthaven.
\newblock An artificial neural network as a troubled-cell indicator.
\newblock {\em Journal of Computational Physics}, 367:166--191, 2018.

\bibitem{Reuther1995}
J.~Reuther and A.~Jameson.
\newblock Aerodynamic shape optimization of wing and wing-body configurations
  using control theory.
\newblock In {\em 33rd Aerospace Sciences Meeting and Exhibit}. American
  Institute of Aeronautics and Astronautics, Jan. 1995.

\bibitem{Samareh2001}
J.~A. Samareh.
\newblock Survey of shape parameterization techniques for high-fidelity
  multidisciplinary shape optimization.
\newblock {\em {AIAA} Journal}, 39(5):877--884, May 2001.

\bibitem{SS1}
C.~Schillings, S.~Schmidt, and V.~Schulz.
\newblock Efficient shape optimization for certain and uncertain aerodynamic
  design.
\newblock {\em Computers and Fluids}, 46:78--87, 2011.

\bibitem{AL}
B.~Settles.
\newblock {\em Active learning: synthesis lectures on artificial intelligence
  and machine learning}.
\newblock Morgan and Claypool, 2012.

\bibitem{GA}
S.~N. Sivanandam and S.~N. Deepa.
\newblock {\em Introduction to genetic algorithms}.
\newblock Springer, 2008.

\bibitem{INC}
J.~Tompson, K.~Schlachter, P.~Sprechmann, and K.~Perlin.
\newblock Accelerating eulerian fluid simulation with convolutional networks.
\newblock In {\em Proceedings of the 34th International Conference on Machine
  Learning, {ICML}}, volume~70, pages 3424--3433, 2017.

\bibitem{TRL1}
F.~Troltzsch.
\newblock {\em Optimal control of partial differential equations}.
\newblock AMS, 2010.

\bibitem{2020SciPy-NMeth}
P.~{Virtanen}, R.~{Gommers}, T.~E. {Oliphant}, M.~{Haberland}, T.~{Reddy},
  D.~{Cournapeau}, E.~{Burovski}, P.~{Peterson}, W.~{Weckesser}, J.~{Bright},
  S.~J. {van der Walt}, M.~{Brett}, J.~{Wilson}, K.~{Jarrod Millman},
  N.~{Mayorov}, A.~R.~J. {Nelson}, E.~{Jones}, R.~{Kern}, E.~{Larson},
  C.~{Carey}, {\.I}.~{Polat}, Y.~{Feng}, E.~W. {Moore}, J.~{Vand erPlas},
  D.~{Laxalde}, J.~{Perktold}, R.~{Cimrman}, I.~{Henriksen}, E.~A. {Quintero},
  C.~R. {Harris}, A.~M. {Archibald}, A.~H. {Ribeiro}, F.~{Pedregosa}, P.~{van
  Mulbregt}, and S.~.~. {Contributors}.
\newblock {SciPy 1.0: Fundamental Algorithms for Scientific Computing in
  Python}.
\newblock {\em Nature Methods}, 17:261--272, 2020.

\end{thebibliography}

\end{document}